\documentclass[a4paper, oneside]{amsart}

\usepackage[utf8]{inputenc}
\usepackage{amsmath}
\numberwithin{equation}{section}
\numberwithin{figure}{section}
\numberwithin{table}{section}
\usepackage{graphicx}
\usepackage{amssymb}
\usepackage{mathrsfs}
\usepackage{amsfonts}
\usepackage[colorlinks=true, allcolors=blue]{hyperref}
\usepackage{subfigure} 
\usepackage{color}
\usepackage{extarrows}
\usepackage{booktabs}
\usepackage{hyperref}
\usepackage{tikz}
\usepackage{extarrows}
\usepackage{booktabs}
\usepackage{amsthm}
\usepackage{enumerate}
\usepackage{enumitem} 
\usepackage{indentfirst}
\usepackage{embrac}
\usetikzlibrary{positioning,shapes}
\usepackage{setspace}
\usepackage{gbt7714}

\def\cC{\mathcal{C}}

\def\bR{{\mathbb R}}

\def\sE{{\mathscr E}}
\def\sF{{\mathscr F}}
\def\sA{{\mathscr A}}

\def\cF{\sF}

\def\sL{\mathscr{L}}
\def\cD{\mathcal{D}}

\def\bs{\mathbf{s}}
\def\br{\mathbf{r}}
\def\bN{\mathbb{N}}

\def\fm{\mathfrak{m}}

\def\${|\!|\!|}
\def\l|{\left|\!\left|\!\left|}
\def\r|{\right|\!\right|\!\right|}

\newtheorem{theorem}{Theorem}[section]
\newtheorem{lemma}[theorem]{Lemma}

\newtheorem{corollary}[theorem]{Corollary}
\theoremstyle{definition}
\newtheorem{definition}[theorem]{Definition}
\newtheorem{example}[theorem]{Example}

\theoremstyle{remark}
\newtheorem{remark}[theorem]{Remark}
\numberwithin{equation}{section}

\begin{document}
\markboth{Resolvent approach to diffusions with discontinuous scale}{Resolvent approach to diffusions with discontinuous scale}

\title[Resolvent approach to diffusions with discontinuous scale]{Resolvent approach to diffusions with discontinuous scale}

\author{Liping Li}
\address{Fudan University, Shanghai, China.  }
\email{liliping@fudan.edu.cn}
\thanks{The first named author is a member of LMNS,  Fudan University.  He is partially supported by NSFC (No.12371144).}

\author{Ying Li$^*$}
\address{School of Mathematics and Computational Science, Xiangtan University, Xiangtan, China.}
\email{liying@xtu.edu.cn}
\thanks{$^*$Corresponding author. She is partially supported by the Youth Project of Hunan Provincial Natural Science Foundation of China (No. 2024JJ6417) and the Project of Education Department of Hunan Province (No. 24B0187).}





\begin{abstract}
	Quasidiffusion is an extension of regular diffusion which can be described as a Feller process on $\mathbb{R}$ with infinitesimal operator $\sL=\frac{1}{2}D_\fm D_\bs$. Here, $\bs(x)=x$ and $\fm$ refers to the (not necessarily fully supported) speed measure. In this paper, we will examine an analogous operator where the scale function $\bs$ is general and only assumed to be non-decreasing. We find that, like regular diffusion or quasidiffusion, the reproducing kernel can still be generated by two specific positive monotone solutions of the $\alpha$-harmonic equation $\sL f=\alpha f$ for each $\alpha>0$. Our main result shows that this reproducing kernel is able to induce a Markov process, which is identical to that obtained in \cite{S79} using a semigroup approach or in \cite{L23} through Dirichlet forms. Further investigations into the properties of this process will be presented.
\end{abstract}
\maketitle
\section{Introduction}\label{SEC1}

Quasidiffusion is a strong Markov process that shares many similarities with diffusion process on a line. It arises from a spectral theory initialized by Kac and Krein, known as Krein correspondence, in \cite{KK74}. Its applications and generalizations in Markov processes and stochastic analysis have been discussed by several authors, such as Kasahara, Kotani, Watanabe, K\"uchler, Sch\"utze, among others. These discussions can be found in, e.g., \cite{W74, K75, S79, KW82, BK87}. In the aforementioned research, the analytic characterization of quasidiffusion and related processes are mostly done within the framework of Feller processes, whose analytic foundation is the theory of strongly continuous semigroups on Banach spaces consisting of continuous functions. Thanks to the symmetry of quasidiffusion, the first author of the current paper has recently utilized Dirichlet form theory to consider an alternative analytic characterization of these processes in several articles including \cite{L23, L23b}. The Dirichlet form is a closed quadratic form on an $L^2$-space that satisfies the Markov property, and its relevant theory can be found in, e.g., the monographs \cite{CF12, FOT11}.

The main objective of this article is to make the representation for the resolvents of quasidiffusion and its extensions simultaneously within the frameworks of both Feller and $L^2$. Roughly speaking, the resolvent is the Laplace transform of the transition function of a Markov process. The study of the resolvents of Markov processes is an important and classical problem. In fact, It\^o and McKean extensively investigated the resolvent theory of diffusion processes in their seminal work \cite{IM74}. Similar studies can also be found in \cite{M68}. However, it was not until 2014 that Fukushima \cite{F14}, with the help of Dirichlet form theory, clarified the resolvents of diffusion processes satisfying general boundary conditions under the assumption of symmetry in both Feller and $L^2$-frameworks. There have also been profound studies on the resolvents of quasidiffusions, such as \cite{K75, KW82, K80}. Krein's correspondence is actually their resolvent theory in the $L^2$-framework. Ogura \cite{O89} focused on the similar topic for the case of Dirichlet boundaries, where the corresponding process solely exhibits absorbing behavior at the boundaries. The second named author of this paper also attempted to characterize the resolvents of a class of generalizations for quasidiffusion in her recent work \cite{LM20}. This paper will build upon that foundation to complete this research.

A quasidiffusion typically arises in terms of a pair $(\bs,\fm)$, where $\bs(x)=x$ and $\fm$ is a positive Radon measure on $\bR$ known as the \emph{speed measure}. (Assume without loss of generality that $\fm(\{0\})=0$.) It is a Feller process corresponding to the infinitesimal operator $\sL=\frac{1}{2}D_\fm D_\bs$. Here, given a function $f$ on $\bR$,  $g=D_\fm D_\bs f$ means that $g\in L^1_\text{loc}(\bR, \fm)$ and there exist two numbers $a,b\in \bR$ such that
\begin{equation}\label{eq:12}
	f(x)=a+bx+\int_0^{x-} (x-z)g(z)\fm(dz),\quad x\in \bR.
\end{equation}
Note that every $f$ in \eqref{eq:12} is continuous on $\bR$ and linear on intervals where $\fm$ does not charge. The right or left $\bs$-derivative of $f$ exists and can be expressed as
\begin{equation}\label{eq:12-2}
	D_\bs^\pm f(x):=\lim_{\varepsilon\downarrow 0}\frac{f(x\pm \varepsilon)-f(x)}{\pm \varepsilon}=b+\int_0^{x\pm}g d\fm,\quad x\in \bR.
\end{equation}
A more detailed summary to quasidiffusion and its infinitesimal operator will be presented in Section~\ref{SEC2}.
When $\fm$ is fully supported, such a quasidiffusion reduces to a regular diffusion process (on its natural scale) on $\bR$ in the sense of, e.g., \cite[VII, \S3]{RY99}.  Similar to regular diffusion, the resolvent of the quasidiffusion is obtained through solving the $\alpha$-harmonic equation
\begin{equation}\label{eq:11}
	\frac{1}{2}D_\fm D_\bs f=\alpha f
\end{equation}
for each $\alpha>0$. The density function of the resolvent with respect to $\fm$ is usually called the \emph{reproducing kernel} due to a representation result in some Hilbert space; see, e.g., \cite{F14}. It can be expressed as the product of two particular positive monotone solutions of \eqref{eq:11}, as stated in Definition~\ref{DEF25}. These two monotone solutions, one increasing and one decreasing, are determined by the behaviour of the process at the left and right boundary points, respectively. This representation of resolvents for quasidiffusion is identical to that of regular diffusion established by Fukushima in \cite{F14}. (Note that this was already observed by It\^o in \cite[Chapter 5]{I06}.) A detailed explanation is provided in Section 2.

We are interested in the generalization of quasidiffusion provided by a more general pair $(\bs,\fm)$, where $\bs$ is only assumed to be a non-decreasing function. Some non-restrictive assumptions will be stated in Definition~\ref{DEF31}. The operator $\sL=\frac{1}{2}D_\fm D_\bs$ is still well defined, but \eqref{eq:12} should be replaced by
\[f(x)=a+b(\bs(x)-\bs(0))+\int_0^{x-} (\bs(x)-\bs(z))g(z)\fm(dz),\quad x\in \bR.
\]
This generalization has been investigated in the context of Markov processes by Sch\"utz \cite{S79}, in the context of sample paths and limit theorem by Ogura \cite{O89} and in the context of Dirichlet forms by the first named author \cite{L23}.
Just like in the case of quasidiffusion, two positive monotone solutions of \eqref{eq:11} still yield a reproducing kernel for general $(\bs,\fm)$, as stated in Definition~\ref{DEF38}. However, finding a “nice” Markov process corresponding to this reproducing kernel is not as straightforward.

In this paper, we will solve this problem completely by considering several different cases separately:
\begin{itemize}
	\item[-] $\bs$ is strictly increasing and continuous;
	\item[-] $\bs$ is strictly increasing;
	\item[-] $\bs$ is non-decreasing.
\end{itemize}
It is worth pointing out that for each case the speed measure $\fm$ is not necessarily fully supported.

The first case is the simplest because $\bs$ constitutes a homeomorphism, which allows us to transfer the situation into the case of quasidiffusion. This case will be examined in Section~\ref{SEC4}.

The second case is the most interesting. We will consider it for a fully supported speed measure that has an isolated mass at points where $\bs$ is neither right nor left continuous in Section~\ref{SEC5}, and for a general speed measure in Section~\ref{SEC6}. The assumption on $(\bs,\fm)$ in Section~\ref{SEC5} is the same as that in \cite{S79}, so it is not surprising that the corresponding Markov process of the reproducing kernel is precisely the process obtained in \cite[Theorem 4.1]{S79}, as shown in Theorem~\ref{THM51}. This process is continuous but does not satisfy the strong Markov property if $\bs$ is not continuous. We can obtain a Feller “version” of it using the standard Ray-Knight compactification method. The continuous extension of the original reproducing kernel on the completion of the state space with respect to $\bs$ corresponds to the Ray-Knight compactification of the original Markov process. More information on Ray-Knight compactification can be found in \cite[\S17]{S88} or Appendix B of \cite{L23}.

For the general speed measure in Section~\ref{SEC6}, the consideration is more complicated. The difficulty lies in the fact that we cannot simply consider the restriction of the reproducing kernel on the topological support of $\fm$. This is because $\bs$ may have at most countably many troublesome discontinuity points on this topological support, and such restriction operation does not yield a well-defined reproducing kernel. To address this issue, we make adjustments to the restricted reproducing kernel at these discontinuity points. It is important to note that this adjustment method is not unique. The adjusted reproducing kernel induces a symmetric, c\`adl\`ag Markov process that enjoys the skip-free property and the quasi-left-continuity, as stated in Theorem~\ref{THM67}. (For the skip-free property, see \cite{BK87} and \cite{L23b}.) Similarly, this process does not always satisfy the strong Markov property. 

In the general case where $\bs$ is only non-decreasing, $\bs$ may be constant on certain intervals. This leads to the reproducing kernel also being constant on these intervals, which is not permissible for generating a Markov process. The standard approach to address this issue is to treat these intervals as a whole and consider them as abstract points in the state space. This operation is referred to as the darning transformation in \cite{L23b}. Necessary explanations will be provided in Section \ref{SEC7}.

In contrast of \cite{O89}, where a similar resolvent approach arose, our paper's primary contribution lies in explicitly providing the resolvent expression for all potential boundary behaviors, encompassing not only absorbing case but also reflection, sojourn, and other phenomena. This significant advancement is highlighted by the careful selection of parameter $\gamma$ in different scenarios, as outlined in Definitions \ref{DEF25} and \ref{DEF38}.

We will now provide a brief explanation of the notations used in this paper. Let $\overline{\mathbb{R}}=[-\infty, \infty]$ be the extended real number system.  A set $E\subset \overline{\bR}$ is called a \emph{nearly closed subset} of $\overline{\bR}$ if $\overline E:= E\cup \{l,r\}$ is a closed subset of $\overline{\bR}$ where $l=\inf\{x: x\in E\}$ and $r=\sup\{x: x\in E\}$.  The point $l$ or $r$ is called the left or right endpoint of $E$.  Denote by $\overline{\mathscr K}$ the family of all nearly closed subsets of $\overline{\bR}$.  Set
\[
	\mathscr K:=\{E\in \overline{\mathscr K}: E\subset \bR\},
\]
and every $E\in \mathscr K$ is called a \emph{nearly closed subset} of $\bR$. If a function $f$ on $E$ is undefined at $j=l$ or $r$, we understand $f(j)$ as the limit $\lim_{x\rightarrow j}f(x)$. In addition, $f(x+)$ (resp. $f(x-)$) stands for the right (resp. left) limit of $f$ at $x$. Denote by $C(E)$ the family of all continuous functions on $E$. Its subfamily consisting of continuous functions with compact support (resp. vanishing at the endpoints not contained in $E$) is denoted by $C_c(E)$ (resp. $C_\infty(E)$).
In Section~\ref{SEC2}, the speed measure $\fm$ is given by an extended real-valued, right continuous, (not necessarily strictly) increasing and non-constant function $m$ on $\bR$. For convenience we also use $m$ to stand for the speed measure. Meanwhile, write $D_mD_x$ (resp. $D^\pm_x$) for $D_\fm D_\bs$ (resp. $D^\pm_\bs$) since $\bs(x)=x$ in this section. The notation $\int_0^{x+}$ (resp. $\int_0^{x-}$) represents the integration over $[0,x]$ (resp. $[0,x)$).


\section{Resolvent approach to quasidiffusions}\label{SEC2}

\subsection{Quasidiffusions}\label{SEC21}
Let $m$ be an extended real valued, right continuous, increasing (not necessarily strictly) and non-constant function on $\bR$.  Set $m(\pm\infty):=\lim_{x\rightarrow \pm\infty}m(x)$. Define
\begin{equation}\label{eq:51-2}
	\begin{aligned}
		 & l_0:=\inf\{x\in \bR: m(x)>-\infty\},\quad r_0:=\sup\{x\in \bR: m(x)<\infty\}, \\
		 & l:=\inf\{x>l_0: m(x)>m(l_0)\},\quad r:=\sup\{x<r_0: m(x)<m(r_0-)\}.
	\end{aligned}\end{equation}
To avoid the trivial case, assume that $l<r$.  Define
\[
	E_m:=\{x\in [l,r]\cap (l_0,r_0): m(x-\varepsilon)<m(x+\varepsilon),\forall \varepsilon>0\}.
\]
One may verify that $E_m\in \mathscr K$ and is ended by $l$ and $r$; see also \cite[Lemma~3.1]{L23b}. Let $I=\langle l,r\rangle$ denote the interval ended by $l$ and $r$, where $l\in I$ (resp. $r\in I$) if and only if $l\in E_m$ (resp. $r\in E_m$). The function $m$ corresponds to a measure on $\bR$, which will continue to be denoted by $m$ if no confusion arises.

Let 
$(W_t)_{t\geq0}$ be a Brownian motion on $\bR$, $(\Omega, \cF^W)$ is the sample space, $(\sF^W_t)$ is the filtration and $\mathbf{P}_x$ is the probability measure on $(\Omega, \cF^W)$ with $\mathbf{P}_x(W_0=x)=1$.  Further, let $\ell^W(t,x)$ be its local time normalized such that for any bounded Borel measurable function $f$ on $\bR$ and $t\geq 0$,
\[
	\int_0^t f(W_s)ds=2\int_\bR \ell^W(t,x)f(x)dx.
\]
Define $S_t:=\int_{\bR} \ell^W(t,x)m(dx)$ for $t\geq 0$ and ($\inf\emptyset:=\infty$)
\[
	\tau_t:=\inf\{u>0: S_u>t\},\quad t\geq 0.
\]
Then $\tau=(\tau_t)_{t\geq 0}$ is a strictly increasing (before $\zeta$ defined as below), right continuous family of $\sF^W_t$-stopping times with $\tau_0=0$,  $\mathbf{P}_x$-a.s.   Define
\begin{equation}\label{eq:22-2}
	\begin{gathered}
		\sF:=\sF^W,\quad \sF_t:=\sF^W_{\tau_t},\quad
		\zeta:=\inf\{t>0: W_{\tau_t}\notin (l_0,r_0)\},\\ X_t:=W_{\tau_t}, \; 0\leq t<\zeta.
	\end{gathered}
\end{equation}
Then $X=\{\Omega, \sF,\sF_t, X_t, (\mathbf{P}_x)_{x\in E_m}\}$,  called \emph{the quasidiffusion with speed measure $m$},  is a standard process with state space $E_m$ and lifetime $\zeta$; see \cite{K86}. We refer the readers to \cite{BG68} for the definition of a standard process and its related terminologies.

Throughout this paper we will assume the following condition to rule out the possibility of killing inside for quasidiffusions:
\begin{itemize}
	\item[(QK)] If $l_0>-\infty$ (resp.  $r_0<\infty$),  then $l=l_0$  (resp.  $r=r_0$).
\end{itemize}
Further explanations can be found in \cite[\S3]{L23b}.

\subsection{Feller's boundary classification for quasidiffusions}\label{SECboundary}

The boundary classification for a quasidiffusion in Feller’s sense is the same as that for a regular diffusion, as mentioned in references such as \cite{I06, F14}. Here, we state the classification specifically for the right endpoint $r$. Set, for $x>0$,
\[
	\sigma(x):=\int_0^x m(y)dy,\quad \lambda(x):=\int_0^{x+} ym(dy),
\]
and $\sigma(r):=\lim_{x\uparrow r}\sigma(x), \lambda(r):=\lim_{x\uparrow r}\lambda(x)$. Then $r$ is called
\begin{itemize}
	\item[(1)] \emph{regular}, if $\sigma(r)<\infty$ and $\lambda(r)<\infty$;
	\item[(2)] \emph{exit}, if $\sigma(r)<\infty$ and $\lambda(r)=\infty$;
	\item[(3)] \emph{entrance}, if $\sigma(r)=\infty$ and $\lambda(r)<\infty$;
	\item[(4)] \emph{natural}, if $\sigma(r)=\lambda(r)=\infty$.
\end{itemize}
In addition, $r$ is called \emph{reflecting} (resp. \emph{absorbing}) if $r$ is regular and $r\in E_m$ (resp. $r\notin E_m$). It is called \emph{instantaneously reflecting} provided that it is reflecting with $m(\{r\})=0$.

For simplification we only consider the case that the left endpoint $l$ is
instantaneously reflecting:
\begin{itemize}
	\item[(L)] $0=l\in E_m$ and $m(0)=m(0-)$.
\end{itemize}
Therefore, for our discussion, we only need to focus on the classification of the right endpoint $r$. It should be noted that the general case can be treated analogously.
Under the assumption (L) along with condition (QK), $I=[l,r\rangle=[0,r\rangle$ and $m(\bR\setminus I)=0$.  Denote the restriction of $m$ to $[0,r)$ still by $m$.

Recall that the left and right derivatives of $f$ with $D_m D_x f=g$ are expressed as \eqref{eq:12-2}. Note that $D_x f(0):=\lim_{x\downarrow 0} D^\pm_x f(x)=b$ and $D_x f(r):=\lim_{x\uparrow r}D^\pm_x f(x)=b+\int_0^{r-}gd\fm$ if $g\in L^1(\bR,m)$.

\subsection{Reproducing kernels of quasidiffusions}\label{SEC22}

With $m$ at hand,  we will formulate the resolvent $(R_\alpha)_{\alpha>0}$ of $X$.  Fix $\alpha>0$.  To begin, we define, for any $x\in [0,r)$,
\[
	u^0(x):\equiv 1,\quad u^{n+1}(x):=\int_0^x \int_0^{y+} u^n(z)dm(z)dy,\quad u(x):=\sum_{n=0}^\infty (2\alpha)^n u^n(x).
\]
By referring to \cite[II \S2, \#2]{M68}, we can see that $u$ is a solution of \eqref{eq:11} and possesses the following properties.

\begin{lemma}
	The function $u$ is a positive, increasing solution of \eqref{eq:11} with $u(0)=1$ and $D_xu(0)=0$.
\end{lemma}
\begin{remark}
	In the representation of the reproducing kernel of $X$ (see Definition~\ref{DEF25}),
	this increasing solution $u$ corresponds to the instantaneously reflecting property of the left endpoint $l$ (assumed by (L)).  When considering a general boundary condition at $l$,  we should choose another positive, increasing solution using a similar argument to the one used to obtain the positive, decreasing solution $v$ as shown below. See also Example~\ref{EXA43}.
\end{remark}

We need another positive, decreasing solution $v$ of \eqref{eq:11}.  To obtain it, define
\[
	u^+(x):=u(x) \int_x^r u(y)^{-2}dy,\quad  u^-(x):=u(x) \int_0^x u(y)^{-2}dy,\quad x\in [0,r).
\]
Then $u^+$ is a decreasing solution of \eqref{eq:11},  while $u^-$ is an increasing solution of \eqref{eq:11} with $u^-(0)=0$ and $D_xu^-(0)=1$;  see,  e.g.,  \cite[II \S2, \#3]{M68}.  The lemma below provides all positive, decreasing solutions of \eqref{eq:11}.

\begin{lemma}\label{LM23}
	Every positive, decreasing solution of \eqref{eq:11} can be represented as
	\begin{equation}\label{eq:21}
		u-\gamma u^-
	\end{equation}
	for a constant $\gamma\in [\underline{\gamma}, \bar{\gamma}]$, up to a multiplicative constant,  where
	\[
		\bar{\gamma}=\left(\int_0^ru(y)^{-2}dy\right)^{-1},\quad \underline{\gamma}=\frac{\bar{\gamma}}{1+\bar{\gamma}/(D_x u(r)u(r))}.
	\]
\end{lemma}
\begin{proof}
	Let $v$ be a positive and decreasing solution of \eqref{eq:11}.
	According to \cite[II,  \S3, \#6]{M68},  there exist constants $c_1$ and $c_2$ such that $v=c_1 u+c_2 u^-$.  Since $v$ is positive,  we have
	\[
		0\leq v(0)=c_1 u(0)+c_2 u^-(0)=c_1.
	\]
	It is evident that $c_1\neq 0$ because otherwise $v=c_2 u^-$ would be either negative or increasing.
	Since $v$ is decreasing,  we obtain
	\[
		0\geq D_xv(0)=c_1 D_x u(0)+c_2 D_xu^-(0)=c_2.
	\]
	Similarly, we have $c_2\neq 0$.
	Hence we may and do assume that $v=u-\gamma u^-$ for some constant $\gamma>0$.  Note that $v$ is positive, decreasing if and only if
	\begin{equation}\label{eq:24}
		\sup_{x}\frac{D_x^+ u(x)}{D_x^+u^-(x)}\leq \gamma\leq \inf_x\frac{u(x)}{u^-(x)}.
	\end{equation}
	Obviously $\inf_x\frac{u(x)}{u^-(x)}=\bar{\gamma}$.  A direct computation gives
	\begin{equation}\label{eq:25}
		\frac{D_x^+ u(x)}{D_x^+u^-(x)}=\left(\int_0^x u^{-2}dy+\frac{1}{u(x)D_x^+u(x)}\right)^{-1}.
	\end{equation}
	On the other hand,  since $D^+_xu^-(x)\geq D_xu^-(0)= 1$,  it follows that $D^+_xu/D^+_x u^-$ has bounded variation.  In addition, according to  \cite[Proposition~2.3]{LM20},
	\[
		\begin{aligned}
			D_m\left(\frac{D_x^+ u}{D_x^+u^-}\right)(x)
			 & =\frac{-D_mD_x u^-(x) D_x^+u(x)+D_x^+ u^-(x)D_mD_x u(x)}{D_x^+u^-(x)D_x^+u^-(x-)}             \\
			 & =\frac{2\alpha}{D_x^+u^-(x)D_x^+u^-(x-)}\left(-u^-(x)D_x^+u(x)+D_x^+u^-(x)u(x)\right)         \\
			 & =\frac{2\alpha}{D_x^+u^-(x)D_x^+u^-(x-)} W(u^-,u)\\
			 &=\frac{2\alpha}{D_x^+u^-(x)D_x^+u^-(x-)} >0.
		\end{aligned}\]
	Thus, $D_x^+ u(x)/D_x^+u^-(x)$ is increasing in $x$.  Particularly,  it follows from \eqref{eq:25} that the lower bound in \eqref{eq:24} is equal to
	\[
		\lim_{x\uparrow r}\frac{D_x^+u(x)}{D_x^+u^-(x)}=\underline{\gamma}.
	\]
	This completes the proof.
\end{proof}
\begin{remark}
	Obviously, we have $\underline{\gamma}\leq \bar{\gamma}$.  Note that $\underline{\gamma}< \bar{\gamma}$ if and only if $r$ is regular.
\end{remark}

Now we are in a position to select the decreasing solution $v$ of \eqref{eq:11} that determines the resolvent $(R_\alpha)_{\alpha>0}$.  When $r$ is not regular,  $\bar{\gamma}=\underline{\gamma}$ and we take
\[
	v:=u-\bar\gamma u^-=\bar\gamma u^+.
\]
When $r$ is regular, we choose $v$ as defined in \eqref{eq:21}, with the following value for $\gamma$:
\[
	\gamma=\left \lbrace
	\begin{aligned}
		 & \bar{\gamma},\qquad\qquad \qquad \qquad \qquad \qquad\qquad\qquad                                & r\text{ is absorbing},  \\
		 & \bar{\gamma}-\frac{\bar{\gamma}^2}{\bar{\gamma}+D_xu(r)u(r)+2\alpha u(r)^2 m(\{r\})},\quad\;\;\; & r\text{ is reflecting}.
	\end{aligned}\right.
\]
Note that $\gamma$ for the reflecting case is equal to $\underline{\gamma}$ if $m(\{r\})=0$. We denote the Wronskian of $u$ and $v$ by
\[
	W:=W(u,v)=D_x^+u(x)v(x)-D_x^+v(x)u(x),
\]
which is independent of $x$; see,  e.g.,  \cite[II \S2, \#5]{M68}.

\begin{definition}\label{DEF25}
	Let $u$ and $v$ be selected as described earlier. The \emph{reproducing kernel} $g_\alpha$,  $\alpha>0$,  of $X$ on $E_m\times E_m$ is defined as
	\[
		g_\alpha(x,y):=\left\lbrace
		\begin{aligned}
			 & W^{-1}u(x)v(y),\quad x\leq y, \\
			 & W^{-1}u(y)v(x),\quad y\leq x,
		\end{aligned}
		\right.
	\]
	where $W^{-1}$ is the reciprocal of the Wronskian $W$.
\end{definition}

\subsection{Resolvent in the $L^2$-framework}

The quasidiffusion is symmetric with respect to $m$,  and its associated Dirichlet form $(\sE,\sF)$ on $L^2(E_m,m)$ is formulated in,  e.g., \cite[Theorem~3.1]{L23}.  The following result,  as an analogue of \cite[Corollary~2.5,  Theorems~2.6 and 5.3]{F14},  presents the resolvent and the generator of $(\sE,\sF)$ in the $L^2$-sense.  The proof is straightforward and follows the arguments in \cite{F14}, so we omit it.

\begin{theorem}\label{THM26}
	Let $(\sE,\sF)$ be the Dirichlet form on $L^2(E_m,m)$ associated to the quasidiffusion $X$.  Then the strongly continuous contraction resolvent of $(\sE,\sF)$ on $L^2(E_m,m)$ is given by
	\[
		G_\alpha f(\cdot)=\int_{E_m} g_\alpha(\cdot,y)f(y)m(dy),\quad \alpha>0, f\in L^2(E_m,m),
	\]
	where $g_\alpha$ is the reproducing kernel of $X$ defined in Definition~\ref{DEF25}.  Furthermore,  the generator $\mathscr A$ with domain $\cD(\mathscr A)$ of $(\sE,\sF)$ on $L^2(E_m,m)$ is expressed as follows:
	\[
		\begin{aligned}
			\cD(\sA)=\{f\in \sF: & D_mD_x f\in L^2(E_m,m), D_x f(0)=0,                                   \\
			                     & D_x f(r)=0\text{ whenever }r\text{ is  instantaneously reflecting}\},
		\end{aligned}\]
	and  for $f\in \cD(\sA)$,
	\[
		\sA f(x)=\left\lbrace
		\begin{aligned}
			 & \frac{1}{2}D_mD_x f(x),\quad x\in E_m\cap (0,r),                           \\
			 & -\frac{D_x f(r)}{2m(\{r\})},\quad\;\;\text{if }x=r\text{ with }m(\{r\})>0.
		\end{aligned} \right.
	\]
\end{theorem}

\subsection{Resolvent in Feller's framework}

Let $E^*:=E_m$ if $r$ is not entrance,  and $E^*:=E_m\cup \{r\}$ if $r$ is entrance.  When $r$ is entrance,  we denote the zero extension of $m$ to $E^*$ as $m$ as well.  Denote by $C(E^*)$ the family of all continuous functions on $E^*$.  Note that every $f\in C(E^*)$ can be treated as a continuous function on $I$,  which is linear on intervals where $m$ is constant.  Furthermore, we define $C_\infty(E^*)$ as the subfamily of $C(E^*)$ consisting of functions that vanish at $r$ if $r\notin E^*$. The definitions of Feller semigroup and Feller process are standard, and we refer the readers to, e.g., \cite{RY99}.

\begin{theorem}\label{THM27}
	The transition function $(P_t)_{t\geq 0}$ of the quasidiffusion $X$ acts on $C_\infty(E^*)$ as a Feller semigroup,  whose resolvent is given by
	\begin{equation}\label{eq:22}
		R_\alpha f(\cdot)=\int_{E^*} g_\alpha(\cdot,y)f(y)m(dy),\quad f\in C_\infty(E^*).
	\end{equation}
	Furthermore,  the infinitesimal operator $\sL$ with domain $\cD(\sL)$ of $(P_t)_{t\geq 0}$ on $C_\infty(E^*)$ is expressed as follows:
	\[
		\begin{aligned}
			\cD(\sL)
			=\{f\in C_\infty(E^*)&:~D_mD_x f\in C_\infty(E^*),  D_xf(0)=0,                               \\
			            &D_x f(r)=0\text{ whenever }r\text{ is instantaneously reflecting}\},
		\end{aligned}\]
	and for  $f\in \cD(\sL)$,
	\[
		\sL f(x)=\left\lbrace
		\begin{aligned}
			 & \frac{1}{2}D_mD_x f(x),\quad x\in E^*\cap [0,r),                             \\
			 & -\frac{D_x f(r)}{2m(\{r\})},\quad\;\;  \text{if }x=r\text{ with }m(\{r\})>0.
		\end{aligned} \right.
	\]
\end{theorem}
\begin{proof}
	The result can be also obtained by mimicking the arguments in \cite[Proposition~6.1]{F14},  while the proof in \cite{F14} does not provide all the details for the case where $r$ is entrance. We will now complement this part. In this case, we have $C_\infty(E^*)=C(E^*)\subset L^2(E^*,m)$.

	Let us prove that $R_\alpha f\in C(E^*)$ for every $f\in C(E^*)$. Consider the first case $f\in C_c(E_m)$. Since $R_\alpha f=G_\alpha f\in \sF\subset C(E_m)$, it suffices to show that $\lim_{x\uparrow r}R_\alpha f(x)$ exists and  is finite.  In fact,
	\[
		\lim_{x\uparrow r} R_\alpha f(x)=W^{-1}\lim_{x\uparrow r} v(x)\int_{E_m} u(y)f(y)m(dy).
	\]
	Note that $v=\bar{\gamma} u^+$ admits a finite limit at $r$.  Hence $R_\alpha f\in C(E^*)$.  When $f\in C_\infty(E_m)$, by taking $f_n\in C_c(E_m)$ with $\|f_n-f\|_\infty\rightarrow 0$ and noting that $R_\alpha f_n\in C(E^*)$,  we can conclude that $R_\alpha f\in C(E^*)$ on account of
	\[
		\alpha \|R_\alpha f_n-R_\alpha f\|_\infty=\alpha \|G_\alpha f_n-G_\alpha f\|_\infty\leq  \|f_n-f\|_\infty.
	\]
	The inequality is due to $1_{E^*}\in L^2(E_m,m)$ and the Markovian property of $\alpha G_\alpha$; see, e.g., \cite[Theorem~1.4.1]{FOT11}.
	For a general $f\in C(E^*)$,  let $f_0:=f-f(r)\cdot 1_{E^*}\in C_\infty(E_m)$. We have
	\[
		R_\alpha f=R_\alpha f_0 + f(r)\cdot R_\alpha 1_{E^*}.
	\]
	Since $R_\alpha f_0\in C(E^*)$ and $R_\alpha 1_{E^*}(x)=G_\alpha 1_{E^*}(x)=1/\alpha$ for any $x\in E_m$,  it follows that $R_\alpha f\in C(E^*)$.

	By virtue of \cite[II \S5, Theorem~3]{M68},  the operator $\sL$ with domain $\cD(\sL)$ is the infinitesimal operator of a certain Feller semigroup,  denoted by $\bar{P}_t$,  on $C(E^*)$.  Denote by $\bar{R}_\alpha$ the resolvent of $\bar{P}_t$.  We have to prove that $\bar{R}_\alpha f=R_\alpha f$ for $f\in C(E^*)$.  Since $\bar{R}_\alpha$ is the resolvent of $\sL$,  it follows that
	\[
		(\alpha-\frac{1}{2}D_mD_x)\bar R_\alpha f=f.
	\]
	On the other hand,  $R_\alpha f=G_\alpha f\in \cD(\sA)$ also yields
	\[
		(\alpha-\frac{1}{2}D_mD_x) R_\alpha f=f.
	\]
	Particularly,  $w:=\bar{R}_\alpha f-R_\alpha f\in C(E^*)$ is a solution of
	\[
		\frac{1}{2}D_mD_x F=\alpha F.
	\]
	In view of \cite[II,  \S4]{M68}, we have $w=C u^+$ for some constant $C$.  On account of $R_\alpha f=G_\alpha f\in \cD(\sA)$ and $\bar{R}_\alpha f\in \cD(\sL)$,  one deduces $D_x w(0)=0$. However, we know that $D_xu^+(0)=-1/u(0)=-1$. This implies that $C=0$, which in turn leads to $w=0$.  As a result, we can conclude that $\bar{R}_\alpha f=R_\alpha f$. This completes the proof.
\end{proof}
\begin{remark}
	When $r$ is entrance,  the state space $E_m$ of $X$ does not include $r$,  thus $X$ is unable to reach $r$ during its lifetime.  However the Feller resolvent \eqref{eq:22} contains information about $r$, and its corresponding Feller process $\bar{X}$ may start from $r$.  Certainly the restriction of $\bar{X}$ to $E_m$ is identified with $X$.  Conversely,  $\bar{X}$ is actually the Ray-Knight compactification of $X$, as explained in \cite[\S4.3]{L23}.
\end{remark}

\section{Reproducing kernels for general pairs}

We now consider a general pair $(\bs, \fm)$ on the interval $I=[0,r\rangle$, where $0$ is included in $I$ and $r$ may or may not be contained in $I$.  For the sake of simplification, we will focus on a pair in the following definition and classify only the right endpoint.

\begin{definition}\label{DEF31}
	The pair $(\bs,\fm)$ on $I=[0,r\rangle$ consists of a non-decreasing function $\bs$ on $I$ and a positive Radon measure $\fm$ on $I$ such that
	\begin{itemize}
		\item[(1)] $\bs(0)=\bs(0+)=0<\bs(x)$ and $\fm(\{0\})=0=\fm(0)<\fm(x)$ for any $x>0$;
		\item[(2)] $\bs(r-\varepsilon)<\bs(r-)=\bs(r)\leq \infty$ and $\fm(r-\varepsilon)<\fm(r-)$ for any $\varepsilon>0$;
		\item[(3)] If $r\in I$,  then $\bs(r)+\fm(r)<\infty$.
	\end{itemize}
	Here and hereafter $\fm$ also stands for the right continuous function on $I$ induced by the measure $\fm$ if no confusion arises.
\end{definition}
\begin{remark}
	The first condition indicates that $0$ is \emph{instantaneously reflecting} with respect to $(\bs,\fm)$.  In fact, if $\bs(x)=x$, this condition is equivalent to the condition (L) in \S\ref{SECboundary}. It loses no generality and is only used to gain simplification. In some specific examples like Examples~\ref{EXA43} and \ref{EXA52}, we will also consider other boundary conditions at the left endpoint.
	The third condition concerns the case that $r$ is reflecting with respect to $(\bs,\fm)$ (see \S\ref{SEC32}).  It is worth pointing out that if $r\in I$, then $\fm$ is a finite measure on $[0,r]$, but $\fm(\{r\})=\fm(r)-\fm(r-)$ may be positive.
\end{remark}

\subsection{Second order derivative}

The objective of this subsection is to introduce a generalized second order differential operator $D_\fm D_\bs$.  Prior to that, we need to establish some notations and terminologies.  A real-valued function $f$ on $[0,r)$ is called \emph{$\bs$-continuous} if $f$ is finite and right continuous  at $0$,  $f$ has finite limits from the left and right at each point in $(0,r)$,  the right or left continuity of $\bs$ at a point implies the same property for $f$, and $f$ is constant on intervals where $\bs$ is constant.  The set of all $\bs$-continuous functions on $[0,r)$ is denoted by $\mathcal C_\bs$.
Put
\[
	\cC_\bs^*:=\{f\in \cC_\bs: f(r):=\lim_{x\uparrow r}f(x)\text{ exists and is finite}\}.
\]
Let $\overline{\bs(I)}$ be the closure of $\bs(I)=\{\bs(x):x\in I\}$ in $\overline{\bR}$.  In view of the second condition of Definition~\ref{DEF31},  $\widehat{r}:=\bs(r)$ is not isolated in $\overline{\bs(I)}$. Furthermore, let $\widehat{I}$ be the set such that $\overline{\bs(I)}\setminus \{\widehat{r}\}\subset \widehat{I}\subset \overline{\bs(I)}$ and that $\widehat{r}\in \widehat{I}$ if and only if $r\in I$. Obviously $\widehat I\in \mathscr{K}$ and hence $[0,\widehat r)\setminus \widehat I$ can be expressed as a union of at most countably many open intervals. We denote by $\widehat{\mathcal C}$ the family of all continuous functions on $[0, \widehat{r})$ that are linear on the open components of  $[0,\widehat{r})\setminus \widehat{I}$.  We define
\[
	\widehat{\cC}^*:=\{\widehat{f}\in \widehat{\cC}: \widehat{f}(\widehat{r}):=\lim_{\widehat{x}\uparrow\widehat{r}}\widehat{f}(\widehat{x})\text{ exists and is finite}\}.
\]
The following lemma is straightforward and the proof is left to the readers.

\begin{lemma}
	There is a linear bijection $T: \mathcal{C}_\bs\rightarrow \widehat{\mathcal{C}}$ such that for any $f\in \mathcal{C}_\bs$,  $\widehat{f}:=Tf$ is determined by
	\[
		\widehat{f}(\bs(x))=f(x),\quad \widehat{f}(\bs(x\pm))=f(x\pm),\quad 0<x<r.
	\]
	Furthermore,  $T$ also forms a linear bijection between $\cC^*_\bs$ and $\widehat{\cC}^*$.
\end{lemma}

The inverse map of $T$ is denoted by $T^{-1}$.
Now we are ready to define the second order derivative with respect to  $(\bs, \fm)$.

\begin{definition}\label{DEF34}
	A real-valued function $f$ on $(0,r)$ has a second order derivative $D_\fm D_\bs f=g$ (with respect to $(\bs,\fm)$) if $g\in L^1_\mathrm{loc}([0,r), \fm)$ and there exist two real numbers $a$ and $b$ such that
	\begin{equation}\label{eq:37}
		f(x)=a+b\bs(x)+\int_0^{x-} (\bs(x)-\bs(y))g(y)\fm(dy),\quad 0<x<r.
	\end{equation}
\end{definition}
\begin{remark}
	By this definition,  if $f$ has the second order derivative $g$, then $f\in \cC_\bs$ with $f(0)=f(0+)=a$.   For $x\in (0,r)$,
	\[
		f(x\pm)=a+b\bs(x\pm)+\int_0^{x\pm}(\bs(x\pm)-\bs(y))g(y)\fm(dy);
	\]
	see also \cite[Proposition~2.2]{S79}.  In addition,  if $\bs(x+\varepsilon)>\bs(x)$ (resp.  $\bs(x-\varepsilon)<\bs(x)$) for any $\varepsilon>0$,  then the right (resp.  left) $\bs$-derivative of $f$ exists and can be defined as
	\[
		D_\bs^\pm f(x):=\lim_{\varepsilon\downarrow 0}\frac{f(x\pm \varepsilon)-f(x)}{\bs(x\pm \varepsilon)-\bs(x)}=b+\int_0^{x\pm}g d\fm.
	\]
	Take a sequence $x_n\uparrow r$ with $\bs(x_n+\varepsilon)>\bs(x_n)$ (resp. $\bs(x_n-\varepsilon)<\bs(x_n)$) for any $\varepsilon>0$ and all $n$ large enough. (Such a sequence exists due to the second condition in Definition~\ref{DEF31}.) When $g\in L^1([0,r),\fm)$, we have
	\begin{equation}\label{eq:34}
		D_\bs f(r):=\lim_{n\rightarrow \infty}D^\pm_\bs f(x_n)=b+\int_0^{r-}gd\fm.
	\end{equation}
	Note that the limit \eqref{eq:34} does not depend on the choice of the sequence $x_n\uparrow r$. Taking an analogous sequence $x_n\downarrow 0$, one has $D_\bs f(0):=\lim_{n\rightarrow \infty}D^\pm_\bs f(x_n)=b$.

\end{remark}

Denote by $\widehat{\fm}$ the image measure of $\fm$ under the map $\bs: I\rightarrow \widehat{I}$.  Note that $\widehat{\fm}$ is a positive Radon measure on $\widehat{I}$ such that $\widehat{\fm}(\widehat{I}\setminus \bs(I))=0$.  The topological support $\text{supp}[\widehat{\fm}]$ of $\widehat{\fm}$ is not necessarily equal to $\widehat{I}$. For $g\in L^1_\mathrm{loc}([0,r),\fm)$,  we define a function $\tilde{g}\in L^1_\mathrm{loc}([0, \widehat{r}),\widehat{\fm})$ as follows: Set $\tilde{g}(\widehat{x}):=g(x)$ for $\widehat{x}=\bs(x)\in \bs(I)$ such that $\bs^{-1}(\widehat{x}):=\{y\in [0,r): \bs(y)=\widehat{x}\}$ is a singleton; if $\bs^{-1}(\widehat{x})$ is not a singleton and $\fm(\bs^{-1}(\widehat{x}))>0$,  set
		\[
			\tilde{g}(\widehat{x}):=\frac{\int_{\bs^{-1}(\widehat{x})}gd\fm}{\fm(\bs^{-1}(\widehat{x}))};
		\]
		for the remaining cases,  let $\tilde{g}(\widehat{x}):=0$.  It is easy to verify $\tilde{g}\in L^1_\mathrm{loc}([0, \widehat{r}),\widehat{\fm})$.  Particularly,  if $g\in \cC_\bs$,  then $\tilde{g}=Tg$,  $\fm$-a.e.  The following lemma is straightforward.

		\begin{lemma}\label{LM36}
			Let $g\in L^1_\mathrm{loc}([0,r),\fm)$ and $\tilde{g}$ be defined as above.  Then the second order derivative of $f\in \cC_\bs$ on $(0,r)$ is $g$,  i.e.  $D_\fm D_\bs f=g$,  if and only if $\widehat{f}=Tf\in \widehat{\cC}$ satisfies
			\begin{equation}\label{eq:31}
				\widehat{f}(\widehat{x})=\widehat a+ \widehat b\widehat{x}+\int_0^{\widehat{x}-} (\widehat{x}-y)\tilde{g}(y)\widehat{\fm}(dy),\quad \forall \widehat{x}\in (0,\widehat{r})
			\end{equation}
			for some real numbers $\widehat a$ and $\widehat b$.
			If $g\in \cC_\bs$,  then \eqref{eq:31} also holds with $Tg$ in place of $\tilde{g}$.
		\end{lemma}
		\begin{remark}
			In the notation of \S\ref{SEC1},  \eqref{eq:31} can be expressed as $D_{\widehat{\fm}} D_{\widehat{x}} \widehat{f}=\tilde{g}$. This particularly implies $\widehat a=\widehat f(0)=f(0)$ and $\widehat b=D_{\widehat{x}} \widehat f(0)=D_\bs f(0)$.
		\end{remark}

		\subsection{Feller's boundary classification}\label{SEC32}

		As in the case of quasidiffusion discussed in \S\ref{SEC2},  we may classify $\widehat{r}$ (resp. $r$) with respect to $\widehat{\fm}$ (resp. $(\bs,\fm)$) in Feller's sense.  Specifically, for $\widehat{x}\in (0,\widehat{r})$, we define
		\[
		{\widehat{\sigma}}(\widehat{x}):=\int_0^{\widehat{x}} {\widehat{\fm}}\left((0, {\widehat{y}}] \right) d{\widehat{y}}, \quad
	{\widehat{\lambda}}({{\widehat{x}}}):=\int_0^{\widehat{x}+} {\widehat{y}} {\widehat{\fm}}(d{\widehat{y}}).
	\]
	By convention, we set ${\widehat{\sigma}}(\widehat{r}):=\lim_{{\widehat{x}}\rightarrow \widehat{r}}{\widehat{\sigma}}({\widehat{x}})$ and ${\widehat{\lambda}}(\widehat{r}):=\lim_{{\widehat{x}}\rightarrow \widehat{r}}{\widehat{\lambda}}({\widehat{x}})$.

	\begin{definition}
		With respect to $\widehat{\fm}$,  the endpoint $\widehat{r}$  is called
		\begin{itemize}
			\item[(1)] \emph{regular},  if ${\widehat{\sigma}}({\widehat{r}})<\infty,  {\widehat{\lambda}}({\widehat{r}})<\infty$;
			\item[(2)] \emph{exit},  if ${\widehat{\sigma}}({\widehat{r}})<\infty,  {\widehat{\lambda}}({\widehat{r}})=\infty$;
			\item[(3)] \emph{entrance},  if ${\widehat{\sigma}}({\widehat{r}})=\infty,  {\widehat{\lambda}}({\widehat{r}})<\infty$;
			\item[(4)] \emph{natural},  if ${\widehat{\sigma}}({\widehat{r}})= {\widehat{\lambda}}({\widehat{r}})=\infty$.
		\end{itemize}
		When $\widehat{r}$ is regular,  we call it \emph{reflecting} (resp.  absorbing) if $\widehat{r}\in \widehat{I}$ (resp.  $\widehat{r}\notin \widehat{I}$). Moreover, when $\widehat{r}$ is reflecting,  we call it \emph{instantaneously reflecting} if $\widehat{\fm}(\{\widehat{r}\})=0$.
	\end{definition}

	Based on this definition,  $r$ is called \emph{regular,  exit,  entrance,  natural,  (instan-\\taneously) reflecting} or \emph{absorbing} with respect to $(\bs,\fm)$ if so is $\widehat{r}$ with respect to $\widehat{\fm}$.

	\subsection{Homogeneous equations and reproducing kernel}\label{SEC33}

	Fix $\alpha>0$. According to Definition~\ref{DEF34}, we have the following homogeneous equation:
	\begin{equation}\label{eq:32}
		\frac{1}{2}D_\fm D_\bs f=\alpha f.
	\end{equation}
	A function $f$ is called a solution of \eqref{eq:32} if $f\in \cC_\bs$ and its second order derivative is $2\alpha f$.  A solution $f$ is \emph{positive} if $f(x)\geq 0$ for any $x\in I$, and is \emph{increasing} (resp. \emph{decreasing}) if $\widehat{f}:=Tf$ is an increasing (resp. decreasing) function on $[0,\widehat{r})$.

	Mimicking \S\ref{SEC22},  we will introduce three special solutions $u,  u^\pm$ of \eqref{eq:32}.  To do this, we set, for any $\widehat x\in [0,\widehat{r})$,
\[
	\widehat u^0(\widehat x):\equiv 1,\quad \widehat u^{n+1}(\widehat x):=\int_0^{\widehat x} \int_0^{\widehat y+} \widehat u^n(\widehat z)d\widehat \fm(\widehat z)d\widehat y,\quad \widehat u(\widehat x):=\sum_{n=0}^\infty (2\alpha)^n \widehat u^n(\widehat x)
\]
and
\[
	\widehat u^+(\widehat x):=\widehat u(\widehat x) \int_{\widehat x}^{\widehat r} \widehat u(\widehat y)^{-2}d\widehat y,\quad  \widehat u^-(\widehat x):=\widehat u(\widehat x) \int_0^{\widehat x} \widehat u(\widehat y)^{-2}d\widehat y.
\]
Then $\widehat u, \widehat u^\pm\in \widehat{\cC}$ are solutions of
\begin{equation}\label{eq:33}
	\frac{1}{2}D_{\widehat{\fm}}D_{\widehat{x}} \widehat{f}=\alpha \widehat{f}.
\end{equation}
Particularly,  $\widehat{u}, \widehat{u}^-$ are increasing while $\widehat{u}^+$ is decreasing. Furthermore, let
\begin{equation}\label{eq:35}
	u:=T^{-1} \widehat{u},\quad u^\pm:=T^{-1} \widehat{u}^\pm.
\end{equation}
According to Lemma~\ref{LM36}, $u$ and $u^-$ are two positive increasing solutions of \eqref{eq:32}, and $u^+$ is a positive decreasing solution of \eqref{eq:32}. The Wronskian of two solutions $\widehat{f}_1$ and $\widehat{f}_2$ of \eqref{eq:33} is denoted by $\widehat{W}(\widehat{f}_1,\widehat{f}_2)$.  For two solutions $f_1$ and $f_2$ of \eqref{eq:32},  their Wronskian is defined as $W(f_1,f_2):=\widehat{W}(Tf_1, Tf_2)$.

When $r$ is regular, both $u(r):=\lim_{x\uparrow r}u(x)$ and $D_\bs u(r)$ defined as \eqref{eq:34} are finite.  Let
\[
	\bar{\gamma}:=\left(\int_0^{\widehat{r}}\widehat{u}(\widehat{x})^2d\widehat{x}\right)^{-1},\quad \underline{\gamma}:=\bar{\gamma}/\left(1+\bar{\gamma}/u(r)D_\bs u(r)\right).
\]
In view of Lemmas~\ref{LM23} and \ref{LM36},  every positive and decreasing solution of \eqref{eq:32} must be of the form $u-\gamma \cdot u^-$ for some constant $\gamma\in [\underline{\gamma}, \bar{\gamma}]$ up to a multiplicative positive constant.  We take
\begin{equation}\label{eq:36}
	v:=u-\gamma\cdot u^-
\end{equation}
with the constant
\[
	\gamma=\left \lbrace
	\begin{aligned}
		 & \bar{\gamma},\qquad\qquad \qquad \qquad \qquad \qquad\qquad\qquad\;\;\;    r\text{ is absorbing},                        \\
		 & \bar{\gamma}-\frac{\bar{\gamma}^2}{\bar{\gamma}+D_\bs u(r)u(r)+2\alpha u(r)^2 \fm(\{r\})},\quad  r\text{ is reflecting}.
	\end{aligned}\right.
\]
When $r$ is not regular, $\bar\gamma=\underline{\gamma}$ and we single out $v:=u-\bar{\gamma} u^-=\bar\gamma u^+$.
The reproducing kernel with respect to $(\bs,\fm)$ is then derived as follows.

\begin{definition}\label{DEF38}
	For every $\alpha>0$,  let $u,  u^\pm$ be defined in \eqref{eq:35}.  Set $v:=\bar \gamma u^+$ when $r$ is not regular,  and otherwise take $v$ to be \eqref{eq:36}.  Then the family of functions $g_\alpha$,  $\alpha>0$,  on $I\times I$ defined as
	\[
		g_\alpha(x,y):=\left\lbrace
		\begin{aligned}
			 & W^{-1}u(x)v(y),\quad x\leq y, \\
			 & W^{-1}u(y)v(x),\quad y\leq x
		\end{aligned}
		\right.
	\]
	is called the \emph{reproducing kernel} with respect to $(\bs,\fm)$,  where $W:=W(u,v)$ is the Wronskian of $u$ and $v$.
\end{definition}

The purpose of the subsequent sections is to examine whether the family of maps:
\[
	R_\alpha: \mathcal{B}^+_b(I)\ni  f\mapsto R_\alpha f(x):=\int_I g_\alpha(x,y)f(y)\fm(dy),\quad \alpha>0
\]
is the resolvent of a certain Markov process,  where $\mathcal{B}^+_b(I)$ consists of all positive and bounded Borel measurable functions on $I$.

\section{Strictly increasing, continuous scale function}\label{SEC4}

In this section, we consider a pair $(\bs,\fm)$ in Definition~\ref{DEF31} with an additional condition that $\bs$ is strictly increasing and continuous.
In particular, $\bs$ is a homeomorphism between $I$ and $\widehat{I}$, and hence $\widehat{I}$ is an interval ended by $0$ and $\widehat{r}$. We denote the topological support of $\fm$ on $I$ and the topological support of $\widehat{\fm}$ on $\widehat{I}$ as $E:=\text{supp}[\fm]$ and $\widehat{E}:=\text{supp}[\widehat{\fm}]$ respectively.  Note that the restriction of $\bs$ to $E$,  still denoted by $\bs$,  is also a homeomorphism between $E$ and $\widehat{E}$.  Hence the reproducing kernel $g_\alpha$ defined in Definition~\ref{DEF38} becomes a homeomorphic image of the reproducing kernel of the quasidiffusion on $\widehat{E}$ with speed measure $\widehat{\fm}$.  In view of Theorems~\ref{THM26} and \ref{THM27},  we readily have the following.

\begin{theorem}
	Let $(\bs,\fm)$ be the pair in Definition~\ref{DEF31} such that $\bs$ is strictly increasing and continuous. The associated reproducing kernel $g_\alpha$,  $\alpha>0$,  is defined in Definition~\ref{DEF38}. Additionally,  let $\widehat{X}$ be the quasidiffusion on $\widehat{E}$ with speed measure $\widehat{\fm}$.  Define the image Markov process $X$ on $E$: $X_t:=\bs^{-1}(\widehat{X}_t)$ for any $t\geq 0$.  Then the resolvent of the transition function of $X$ is
	\[
		R_\alpha f(\cdot ):=\int_E g_\alpha(\cdot,y)f(y)\fm(dy),\quad \alpha>0,  f\in \mathcal{B}^+_b(E).
	\]
	Furthermore, the following properties hold:\
	\begin{itemize}
		\item[(1)] $X$ is an $\fm$-symmetric Markov process on $E$ whose associated Dirichlet form $(\sE,\sF)$ on $L^2(E,\fm)$ is regular and irreducible,  whose strongly continuous contraction resolvent on $L^2(E,\fm)$ is
			\[
				G_\alpha f(x)=\int_E g_\alpha(\cdot, y)f(y)\fm(dy),\quad \alpha>0,  f\in L^2(E,\fm),
			\]
			and whose generator $\sA$ with domain $\cD(\sA)$ on $L^2(E,\fm)$ is expressed as follows:
			\[
				\begin{aligned}
					\cD(\sA)=\{f\in \sF: & D_\fm D_\bs f\in L^2(E,\fm), D_\bs f(0)=0,                              \\
					                     & D_\bs f(r)=0\text{ whenever }r\text{ is  instantaneously reflecting}\},
				\end{aligned}\]
			and  for $f\in \cD(\sA)$,
			\[
				\sA f(x)=\left\lbrace
				\begin{aligned}
					 & \frac{1}{2}D_\fm D_\bs f(x),\quad     &  & x\in E\cap (0,r),                       \\
					 & -\frac{D_\bs f(r)}{2\fm(\{r\})},\quad &  & \text{if }x=r\text{ with }\fm(\{r\})>0.
				\end{aligned} \right.
			\]
		\item[(2)] Set $E^*:=E$ if $r$ is not entrance and $E^*:=E\cup \{r\}$ if $r$ is entrance.  Denote the zero extension of $\fm$ to $E^*$ still by $\fm$.  Then the transition function of $X$ acts on $C_\infty(E^*)$ as a Feller semigroup whose resolvent is
			\[
				R_\alpha f(\cdot)=\int_{E^*}g_\alpha(\cdot, y)f(y)\fm(dy),\quad \alpha>0, f\in C_\infty(E^*),
			\]
			and whose infinitesimal operator $\sL$ with domain $\cD(\sL)$ on $C_\infty(E^*)$ is expressed as follows:
			\[
				\begin{aligned}
					\cD(\sL)=\{ & f\in C_\infty(E^*):  D_\fm D_\bs f\in C_\infty(E^*),  D_\bs f(0)=0,    \\
					            & D_\bs f(r)=0\text{ whenever }r\text{ is instantaneously reflecting}\},
				\end{aligned}\]
			and for  $f\in \cD(\sL)$,
			\[
				\sL f(x)=\left\lbrace
				\begin{aligned}
					 & \frac{1}{2}D_\fm D_\bs f(x),\quad     &  & x\in E^*\cap [0,r),                      \\
					 & -\frac{D_\bs f(r)}{2\fm(\{r\})},\quad &  & \text{if }x=r \text{ with }\fm(\{r\})>0.
				\end{aligned} \right.
			\]
	\end{itemize}
\end{theorem}
\begin{remark}\label{RM42}
	A pathwise approach to the process $X$ was established in \cite{L23b}. It particularly demonstrates that $X$ satisfies three important properties: the \emph{regular property},  the \emph{skip-free property}, and having \emph{no killing inside}.  Here the skip-free property means $(X_{t-} \wedge X_t,  X_{t-}\vee X_t)\cap E=\emptyset$ for any $t<\zeta$,  where $\zeta$ is the lifetime of $X$.  The explicit expression of the Dirichlet form $(\sE,\sF)$ has been formulated in, e.g., \cite[Theorem~5.3]{L23b}.  The irreducibility of $(\sE,\sF)$ has been proved in \cite[Theorem~3.1]{L23}.
\end{remark}

A typical example is the (symmetric) birth and death processes on $\bN$ (or on $\bN_\infty:=\bN\cup \{\infty\}$); see,  e.g., \cite[\S7]{L22}.  It should be noted that the left endpoint $0$ does not satisfy the first condition in Definition~\ref{DEF31} for birth and death processes.  However, this provides an opportunity to explain how to formulate the reproducing kernel for the case with a general left boundary (together with a general right boundary).

\begin{example}\label{EXA43}
	We are given a \emph{conservative} and \emph{totally stable} birth-death density matrix
	\begin{equation}\label{eq:38}
		Q=(q_{ij})_{i,j\in \bN}:=\left(\begin{array}{ccccc}
				-q_0   & b_0    & 0      & 0      & \cdots \\
				a_1    & -q_1   & b_1    & 0      & \cdots \\
				0      & a_2    & -q_2   & b_2    & \cdots \\
				\cdots & \cdots & \cdots & \cdots & \cdots
			\end{array}  \right),
	\end{equation}
	where $a_k>0$,  $k \geq 1$,  $b_k>0$,  $k\geq 0$ and $q_k=a_k+b_k$ (Set $a_0=0$ for convenience).  A continuous time Markov chain is called a \emph{birth-death $Q$-process} ($Q$-process for short) if its transition matrix $(p_{ij}(t))_{i,j\in \bN}$ is \emph{standard} and its \emph{density matrix} is $Q$,  i.e.  $p'_{ij}(0)=q_{ij}$ for $i,j\in \bN$.  It is called \emph{symmetric} if $\mu_i p_{ij}(t)=\mu_j p_{ji}(t)$ for any $i,j\in \bN$ and $t\geq 0$,  where
	\[
		\mu_0=1,  \quad \mu_k=\frac{b_0b_1\cdots b_{k-1}}{a_1a_2\cdots a_k}, \; k\geq 1.
	\]
	Actually $\mu=\{\mu_k: k\in \bN\}$ is the unique (possible) symmetric measure of a $Q$-process up to a multiplicative constant; see, e.g., \cite[Lemma~6.6]{C04}.

	The pair $(\bs,\fm)$ for a symmetric $Q$-process can be taken as follows: $I:=[0,\infty\rangle$,  $\bs$ is a strictly increasing and continuous function on $I$ such that
	\begin{equation}\label{eq:39}
		\bs(0)=0,\quad \bs(1)=\frac{1}{2b_0},\quad \bs(k)=\frac{1}{2b_0}+\sum_{i=2}^k\frac{a_1a_2\cdots a_{i-1}}{2b_0b_1\cdots b_{i-1}},\; k\geq 2
	\end{equation}
	and $\fm=\mu$. (The candidates of $\bs$ are not unique.)

	In view of \cite[\S7]{L22},  if $\infty$ is not regular with respect to $(\bs,\fm)$,  then the symmetric $Q$-process is unique and identified with the \emph{minimal $Q$-process}.  If $\infty$ is regular,  all symmetric $Q$-processes are parameterized by a constant $\kappa\in [0,\infty]$. The case $\kappa=0$ (resp. $\kappa=\infty$) corresponds to the so-called $(Q,1)$\emph{-process} (resp. the minimal $Q$-process) with the instantaneously reflecting (resp. absorbing) boundary behaviour at $\infty$. For $\kappa\in (0,\infty)$,  the corresponding $Q$-process admits both reflecting and killing at $\infty$.  (It appears a Robin boundary condition at $\infty$.)    Since the current paper does not concern killing inside,  only  the minimal $Q$-process and the $(Q,1)$-process will be examined in what follows.

	To formulate the reproducing kernels of $Q$-processes, we replace the underlying point $0$ with another point $e:=1/2$ in the integrations involved. Specifically, set, for $\widehat{x}\in (0,\widehat{r})$ with $\widehat{r}=\lim_{k\rightarrow \infty}\bs(k)$,
	\[
		\widehat u^0(\widehat x):\equiv 1,\quad \widehat u^{n+1}(\widehat x):=\int_{\widehat{e}}^{\widehat x} \int_{\widehat{e}}^{\widehat y+} \widehat u^n(\widehat z)d\widehat \fm(\widehat z)d\widehat y,\quad \widehat u(\widehat x):=\sum_{n=0}^\infty (2\alpha)^n \widehat u^n(\widehat x),
	\]
	where $\widehat{e}:=\bs(e)$ and $\int_{[\widehat{e},\widehat{x}]}:=-\int_{(\widehat{x},\widehat{e}]}, \int_{[\widehat{e},\widehat{x})}:=-\int_{[\widehat{x},\widehat{e}]}$ for $\widehat{x}<\widehat{e}$. Additionally, let
	\[
		\widehat u^+(\widehat x):=\widehat u(\widehat x) \int_{\widehat x}^{\widehat r} \widehat u(\widehat y)^{-2}d\widehat y,\quad  \widehat u^-(\widehat x):=\widehat u(\widehat x) \int_0^{\widehat x} \widehat u(\widehat y)^{-2}d\widehat y
	\]
	and
	\[
		u:=T^{-1} \widehat{u},\quad u^\pm:=T^{-1} \widehat{u}^\pm.
	\]
	Note that $u^\pm$ are strictly monotone while $u$ is constant on $(0,1)$.  We can obtain the positive decreasing solution $v$ of equation \eqref{eq:32} using a similar approach as in Definition~\ref{DEF38}. Specifically, we define
	\[
		v:=u-\gamma u^-,
	\]
	where $\gamma=\bar\gamma:=\left(\int_0^{\widehat r}\widehat u(y)^{-2}dy\right)^{-1}$ (resp. $\gamma=\underline{\gamma}:=\lim_{\widehat{x}\uparrow \widehat{r}}(D^+_{\widehat x} \widehat{u}(\widehat{x})/D_{\widehat{x}}^+\widehat{u}^-(\widehat{x}))$) for the minimal $Q$-process (resp. the $(Q,1)$-process). However, the positive increasing solution in the representation of the reproducing kernel differs from that in Definition~\ref{DEF38}. Instead, we require a positive increasing solution $\tilde{u}$ such that (see, e.g.,  \cite[(5.27)]{F14})
	\[
		\lim_{x\downarrow 0}\left(-D_\bs^+ \tilde{u}(x)+2\alpha \mu_0\tilde{u}(x)\right)=0.
	\]
	By mimicking the argument of Lemma~\ref{LM23},  one easily gets $\tilde{u}=u+2\alpha \mu_0 u^-=u+2\alpha u^-$ (up to a multiplicative constant).  Eventually the reproducing kernel $g_\alpha$,  $\alpha>0$,  with respect to $(\bs,\fm)$ is given by:
	\begin{align*}
		\quad&g_\alpha(k_1, k_2) \\
		& =W^{-1}\tilde{u}(k_1)v(k_2) \\
		                   & = \left\lbrace
		\begin{aligned}
			 & \frac{\left(u(k_1)+2\alpha u^-(k_1)\right)\left(u(k_2)-\underline{\gamma} u^-(k_2)\right)}{2\alpha+\underline{\gamma}} &  & \text{ for }(Q,1)\text{-process},    \\
			 & \frac{\bar{\gamma} \left(u(k_1)+2\alpha u^-(k_1)\right)u^+(k_2)}{2\alpha+\bar{\gamma}}                                 &  & \text{ for minimal }Q\text{-process}
		\end{aligned}
		\right.
	\end{align*}
	for $k_1\leq k_2$ with $k_1,k_2\in \bN$,  where $W:=W(\tilde{u}, v)$ is the Wronskian of $\tilde{u}$ and $v$. Note that  $g_\alpha(k_2,k_1):=g_\alpha(k_1,k_2)$.

	An alternative treatment of the resolvent representation for $Q$-processes was appeared in \cite{F59}; see also \cite{WY92}.
	Symmetric $Q$-processes are Feller processes on $\bN$ when $\infty$ is exit,  natural or absorbing,  and on $\bN_\infty$ when $\infty$ is reflecting or entrance.  For more considerations on,  e.g.,  non-symmetric $Q$-processes,  we refer the readers to \cite{L22}.
\end{example}

\section{Strictly increasing, discontinuous scale function with fully supported speed measure}\label{SEC5}

The next two sections are devoted to the case that $\bs$ is strictly increasing.
To handle this situation effectively,  it is crucial to introduce two ``inverse" maps of $\bs: I\rightarrow \widehat{I}$.  The first one,  denoted by $\mathbf{r}: \widehat{I}\rightarrow I$,  is a surjective map that merges the gaps of $\widehat{I}$.  Note that
\begin{align*}
	\widehat{I}=\{\bs(x):x\in I\}&\cup\{\bs(x-):\bs(x-)\neq \bs(x), x\in I\}\\
	&\cup \{\bs(x+): \bs(x+)\neq \bs(x), x\in I\}.
\end{align*}
Then $\br$ is defined as $\br(\widehat{x}):=x$ for either $\widehat{x}=\bs(x)$ or $\widehat{x}=\bs(x\pm)$ with $\bs(x\pm)\neq \bs(x)$.  The second is a homeomorphism from $\widehat{I}$ to an \emph{enlarged} space of $I$.  We make the completion of $I$ with respect to the metric
\[
	\rho(x,y):=|\tan^{-1}x-\tan^{-1}y|+|\tan^{-1}\bs(x)-\tan^{-1}\bs(y)|,
\]
but exclude $r$ from this completion if $r\notin I$.  The resulting enlarged space,  denoted by $I^*$,  is obtained heuristically by splitting each discontinuous point of $\bs$ into two or three.  Then $\bs$ can be extended to a homeomorphism $\bs^*: I^*\rightarrow \widehat{I}$ and its inverse $\br^*:\widehat I \rightarrow I^*$ is actually the second wanted ``inverse" of $\bs$.

In this section, we also assume the following condition:
\begin{equation}\label{eq:51-3}
	\text{$\fm$ is fully supported on $I$, and $\fm(\{x\})>0$ whenever $\bs(x\pm)\neq \bs(x)$.}
\end{equation}
This particularly implies that $\widehat{\fm}$ is a fully supported Radon measure on $\widehat{I}$.
Let $\fm^*$ denote the zero extension of $\fm$ to $I^*$.  Clearly,  $\fm^*$ is identified with the image measure of $\widehat{\fm}$ under the homeomorphism $\br^*:\widehat{I}\rightarrow I^*$.

The quasidiffusion on $\widehat{I}$ with speed measure $\widehat{\fm}$ and lifetime $\widehat{\zeta}$ can be denoted as
\[
	\widehat X=\left\{\widehat{\Omega},\widehat{\sF}, \widehat{\sF}_t, \widehat{X}_t, (\widehat{\mathbf{P}}_{\widehat{x}})_{\widehat{x}\in \widehat{I}}\right\},
\]
where $(\widehat{\Omega}, \widehat{\sF})$ is the sample space (of c\`adl\`ag paths),  $\widehat{\sF}_t$ is the filtration,  $\widehat{\mathbf{P}}_{\widehat{x}}$ is the probability measure on $(\widehat{\Omega},\widehat{\sF})$ with $\widehat{\mathbf{P}}_{\widehat{x}}(\widehat{X}_0=\widehat{x})=1$. In what follows,  we will raise two image processes of $\widehat{X}$ with the help of $\br$ and $\br^*$ respectively.
The first one is to define $ \Omega:=\left\{\omega\in \widehat{\Omega}: \widehat{X}_0(\omega)\in \bs(I)\right\}\in \widehat\sF_0, \sF:=\widehat{\sF}\cap \Omega=\{A\cap  \Omega: A\in \widehat{\sF}\}$ and
\begin{equation}\label{eq:53}
	\begin{aligned}
		 & \sF_t:=\widehat{\sF}_t\cap \Omega=\{A\cap  \Omega: A\in \widehat{\sF}_t\},\quad  X_t(\omega):=\br(\widehat{X}_t(\omega)), \; \omega\in \Omega, \\
		 & {\mathbf{P}}_x:=\widehat{\mathbf{P}}_{\bs(x)}|_{ \Omega},\; x\in I, \quad  \zeta(\omega):=\widehat{\zeta}(\omega),\; \omega\in \Omega.
	\end{aligned}
\end{equation}
Clearly, $(\sF_t)_{t\geq 0}$ is a right continuous filtration on $(\Omega, \sF)$,  and $( X_t)_{t\geq 0}$ is a family of random variables on $ \Omega$ adapted to $(\sF_t)_{t\geq 0}$.  The second image process is much simpler since $\br^*$ is a homeomorphism.
It is a standard process on $I^*$ and denoted by $X^*:=(X^*_t)_{t\geq 0}$ for simplification,  where $X^*_t:=\br^*(\widehat{X}_t)$.

\begin{theorem}\label{THM51}
	Let $(\bs,\fm)$ be a pair in Definition~\ref{DEF31} such that $\bs$ is strictly increasing, discontinuous and that \eqref{eq:51-3} holds.  Furthermore, let $g_\alpha$, $\alpha>0$,  be the reproducing kernel with respect to $(\bs,\fm)$ as in Definition~\ref{DEF38}.
	\begin{itemize}
		\item[(1)] The stochastic process
			\[
				X=\{\Omega, \sF, \sF_t, X_t, ({\mathbf{P}}_x)_{x\in I}\}
			\]
			with lifetime $\zeta$ is an $\fm$-symmetric continuous Markov process on $I$,  which fails to have the strong Markov property.  Furthermore,  the resolvent of its transition function is
			\begin{equation}\label{eq:51}
				R_\alpha f(\cdot):=\int_I g_\alpha(\cdot,  y)f(y)\fm(dy),\quad \alpha>0, f\in \mathcal{B}^+_b(I).
			\end{equation}
		\item[(2)] Let $E^*:=I^*$ if $r$ is not entrance and $E^*:=I^*\cup \{r\}$ if $r$ is entrance.  The zero extension of $\fm^*$ to $E^*$ is still denoted by $\fm^*$.  Then $g_\alpha$ can be continuously extended to a function $g_\alpha^*$ on $I^*\times I^*$,  and the transition function of $X^*$ acts on $C_\infty(E^*)$ as a Feller semigroup,  whose resolvent is
			\begin{equation}\label{eq:52}
				R^*_\alpha f^*(\cdot)=\int_{E^*} g^*_\alpha(\cdot,  y)f^*(y)\fm^*(dy),\quad \alpha>0,  f^*\in C_\infty(E^*).
			\end{equation}
			Particularly,  $X^*$ is the canonical Ray-Knight compactification of $X$.
	\end{itemize}
\end{theorem}
\begin{proof}
	We only prove \eqref{eq:51} and \eqref{eq:52}.  The remainder parts of this theorem have already been established in \cite[Theorem~4.5 and \S4.3]{L23}.  For the terminology of (canonical) Ray-Knight compactification, we refer the readers to \cite[\S4.3]{L23}.

	Let $\widehat{E}:=\widehat{I}$ if $\widehat{r}$ is not entrance, otherwise let $\widehat{E}:=\widehat{I}\cup \{\widehat{r}\}$.  Set $\widehat{g}_\alpha(\widehat{x},\widehat{y}):=\widehat{W}^{-1}\widehat{u}(\widehat{x})\widehat{v}(\widehat{y})$ for $\widehat{x},\widehat{y}\in \widehat{I}$ with $\widehat{x}\leq \widehat{y}$ and $\widehat{g}_\alpha(\widehat{y},\widehat{x}):=\widehat{g}_\alpha(\widehat{x},\widehat{y})$,  where $\widehat{u}, \widehat{v}$ are defined in \S\ref{SEC33} and $\widehat{W}:=\widehat{W}(\widehat{u},\widehat{v})$ represents the Wronskian of $\widehat{u}$ and $\widehat{v}$.
	Based on Theorem~\ref{THM27} and $\widehat{\fm}(\widehat{I}\setminus \bs(I))=0$,  the resolvent of the transition function of $\widehat{X}$ is
	\begin{equation}\label{eq:54}
		\widehat{R}_\alpha \widehat{f}(\cdot)=\int_{\widehat{I}}\widehat{g}_\alpha(\cdot, \widehat{y})\widehat{f}(\widehat{y})\widehat{\fm}(d\widehat{y})=\int_{\bs(I)}\widehat{g}_\alpha(\cdot, \widehat{y})\widehat{f}(\widehat{y})\widehat{\fm}(d\widehat{y}),\quad \widehat{f}\in \mathcal{B}^+_b(\widehat{I}).
	\end{equation}
	This,  together with the definition \eqref{eq:53} of $X$,  implies that for $f\in \mathcal{B}^+_b(I)$ and $x\in I$, we have
	\[
		\begin{aligned}
			R_\alpha f(x) & =\int_0^\infty e^{-\alpha t} \mathbf{E}_x f(X_t)dt=\int_0^\infty e^{-\alpha t} \mathbf{\widehat E}_{\bs(x)} f(\br({\widehat X}_t))dt              \\
			              & =\widehat{R}_\alpha (f\circ \br)(\bs(x))=\int_{\bs(I)}\widehat{g}_\alpha(\bs(x), \widehat{y})(f\circ \br)(\widehat{y})\widehat{\fm}(d\widehat{y}) \\
			              & =\int_I g_\alpha(x,y)f(y)\fm(dy).
		\end{aligned}\]
	Hence we arrive at \eqref{eq:51}. Furthermore,  the existence of $g^*_\alpha$ is a consequence of the $\bs$-continuity of $u$ and $v$.  Particularly,  $\widehat{g}_\alpha(\cdot,\cdot)=g^*_\alpha(\br^*(\cdot),\br^*(\cdot))$.  By using the fact that $\br^*: \widehat{I}\rightarrow I^*$ is a homeomorphism, together with equation \eqref{eq:54}, one can easily obtain \eqref{eq:52}.  This completes the proof.
\end{proof}

The example below formulates the explicit expression of the reproducing kernel for a simple discontinuous scale function.


\begin{example}\label{EXA52}
	Let $I=\bR$, $\fm$ be the Lebesgue measure and
	\begin{align*}
		\bs(x)=x,~x<0,\quad \bs(x)=x+2/\kappa,~x\geq0,
	\end{align*}
	where $\kappa>0$ is a given constant. The non-strong-Markov process $X$ in Theorem~\ref{THM51} for this pair has been studied in \cite[\S5.1]{L23}. Its canonical Ray-Knight compactification $X^*$ on $E^*=I^*=(-\infty,0-]\cup [0,\infty)$ is called the \emph{snapping out Brownian motion} with parameter $\kappa$; see, e.g, \cite{L17} and \cite{LS20}. Here $0-$ should be treated as another origin point distinct from $0$.

	Let us compute the reproducing kernel $g_\alpha$ for $\alpha>0$. The three particular solutions $u, u^\pm$ of \eqref{eq:32} arise in a similar argument to that in Example~\ref{EXA43}, while the underlying point for related integrations is replaced by $e=0$ and the left endpoint is $l=-\infty$. The expression for $u$ is actually the same as the analogue for Brownian motion. Specifically,
	\[
		u(x)=\left(e^{\sqrt{2\alpha}x}+e^{-\sqrt{2\alpha}x}\right)/2.
	\]
	Then, it is straightforward to compute that
	\begin{align*}
		u^+(x)
		= \left\lbrace
		\begin{aligned}
			 & \frac{e^{-\sqrt{2\alpha}x}}{\sqrt{2\alpha}}+\frac{e^{\sqrt{2\alpha}x}+e^{-\sqrt{2\alpha}x}}{\kappa}, &  & x<0,    \\
			 & \frac{e^{-\sqrt{2\alpha}x}}{\sqrt{2\alpha}},                                                         &  & x\geq 0
		\end{aligned}
		\right.
	\end{align*}
	and
	\begin{align*}
		u^-(x)
		= \left\lbrace
		\begin{aligned}
			 & \frac{e^{\sqrt{2\alpha}x}}{\sqrt{2\alpha}},                                                         &  & x<0,     \\
			 & \frac{e^{\sqrt{2\alpha}x}}{\sqrt{2\alpha}}+\frac{e^{\sqrt{2\alpha}x}+e^{-\sqrt{2\alpha}x}}{\kappa}, &  & x\geq 0.
		\end{aligned}
		\right.
	\end{align*}
	Note that both $-\infty$ and $\infty$
	are natural boundary.
	Hence the reproducing kernel is
	\[
		g_\alpha(x,y)=\left\lbrace
		\begin{aligned}
			 & \left(\frac{2}{\sqrt{2\alpha}}+\frac{2}{\kappa}\right)^{-1}u^-(x)u^+(y),\quad x\leq y, \\
			 & \left(\frac{2}{\sqrt{2\alpha}}+\frac{2}{\kappa}\right)^{-1}u^-(y)u^+(x),\quad y\leq x.
		\end{aligned}
		\right.
	\]
	It is important to note that although $g_\alpha$ is not continuous on $\bR\times \bR$, it can be extended to a continuous function $g^*_\alpha$ on $I^*\times I^*$.
\end{example}

\section{Strictly increasing, discontinuous scale function with general speed measure}\label{SEC6}

In this section, we will remove the assumption \eqref{eq:51-3} and instead consider a strictly increasing, discontinuous scale function with a general speed measure.
Note that $I^*$,  $\br: \widehat{I}\rightarrow I$,  $\bs^*: I^*\rightarrow \widehat{I}$ and $\br^*: \widehat{I}\rightarrow I^*$ are still well defined.
Denote the topological support of $\fm$ on $I$ and the topological support of $\widehat{\fm}$ on $\widehat{I}$ by $F:=\text{supp}[\fm]\subset I$ and $\widehat{F}:=\text{supp}[\widehat{\fm}]\subset \widehat{I}$ respectively. Additionally, let  $F^*:=\br^*(\widehat{F})$,  which represents the topological support of $\fm^*=\widehat{\fm}\circ \bs^*$ on $I^*$.

\begin{lemma}\label{LM61}
	Assume that $\bs$ is strictly increasing.  Then $F=\br(\widehat{F})$.
\end{lemma}
\begin{proof}
	We first show that $x\notin \br(\widehat{F})$ for $x\in I\setminus F$.  Note that if $x\in I\setminus F$, then there exists a constant $\varepsilon>0$ such that $\fm((x-\varepsilon,x+\varepsilon))=0$.   The definition of $\widehat{\fm}$ yields
	\[
		\widehat{\fm}\left((\bs(x-\varepsilon), \bs(x+\varepsilon)\right)=\fm((x-\varepsilon,x+\varepsilon))=0.
	\]
	Since $\bs$ is strictly increasing,  we have $\bs(x-\varepsilon)<\bs(x-)\leq \bs(x)\leq \bs(x+)<\bs(x+\varepsilon)$.  Thus $\bs(x-),\bs(x),\bs(x+)\notin  \widehat{F}$. This leads to $x\notin \br(\widehat{F})$.

	Conversely,  suppose $x\in I\setminus \br(\widehat{F})$,
	i.e. $\bs(x-),\bs(x),\bs(x+)\notin \widehat{F}$.  Since $$\widehat{\fm}((\bs(x-),\bs(x)))=\widehat{\fm}((\bs(x),\bs(x+)))=0,$$ there exists a constant $\varepsilon>0$ such that
	\[
		\widehat{\fm}\left((\bs(x-)-\varepsilon, \bs(x+)+\varepsilon) \right)=0.
	\]
	Take $\delta>0$ such that $\bs(x-\delta)>\bs(x-)-\varepsilon$ and $\bs(x+\delta)<\bs(x+)+\varepsilon$.  We have
	\[
		\fm((x-\delta, x+\delta))\leq \widehat{\fm}\left((\bs(x-)-\varepsilon, \bs(x+)+\varepsilon) \right)=0.
	\]
	Therefore $x\notin F$.  This completes the proof.
\end{proof}

But in general we do not have $\bs(F)\subset \widehat{F}$. This is revealed in the following example.

\begin{example}\label{EXA62}
	Consider $I=[0,2]$ and  let $\fm$ be the Lebesgue measure on $[0,2]$. Define
	\begin{equation}\label{eq:64}
		\bs(x)=x,\; 0\leq x<1, \quad  \bs(1)=2,\quad \bs(x)=x+2,\; 1<x\leq 2.
	\end{equation}
	We have $F=[0,2]$,  $\bs(F)=[0,1)\cup \{2\}\cup (3,4]$ and $\widehat{F}=[0,1]\cup [3,4]$.  Particularly,  $\bs(F)$ is not a subset of $\widehat{F}$.
\end{example}

Set $\dot F:=\{x\in F: \bs(x)\in \widehat{F}\}$ and $\dot{\widehat{F}}:=\{\bs(x):x\in \dot F\}\subset \widehat{F}$.  Although $\dot{F}$ is not necessarily identical to $F$, we always have the following.

\begin{lemma}\label{LM63}
	$\fm(F\setminus \dot F)=\widehat{\fm}(\widehat{F}\setminus \dot{\widehat{F}})=0$.  Particularly $F$ (resp. $\widehat F$) is the closure of $\dot F$ (resp.  $\dot{\widehat{F}}$) in $I$ (resp. $\widehat{I}$).
\end{lemma}
\begin{proof}
	We will prove the lemma by contradiction. Suppose  $\fm(F\setminus \dot F)>0$.  Then, we have  $\widehat{\fm}(\bs(F)\setminus \widehat{F})\geq \fm(F\setminus \dot F)>0$. This violates the definition of $\widehat{F}$.  Therefore,  $\fm(F\setminus \dot F)=0$.  On the other hand,
	\[
		\widehat{\fm}(\widehat{F}\setminus \dot{\widehat{F}})=\fm(\{x\in I: \bs(x)\in \widehat{F}\setminus \dot{\widehat{F}}\})\leq \fm(I\setminus \dot F)=\fm(I\setminus F)+\fm(F\setminus \dot F)=0.
	\]
	This completes the proof.
\end{proof}
\begin{remark}
	Lemma~\ref{LM61} states that $F\setminus \dot F\subset \{x\in F: \bs(x-)\neq \bs(x)\text{ or }\bs(x)\neq \bs(x+)\}$. Consequently, $F\setminus \dot F$ (as well as $\widehat{F}\setminus \dot{\widehat{F}}$) is a set containing at most countably many points.
\end{remark}

\subsection{Resolvent restricted to $\dot F$}

Recall that the reproducing kernel $g_\alpha$, where $\alpha > 0$, with respect to $(\bs,\fm)$ is defined in Definition~\ref{DEF38}.  The aim of this subsection is to demonstrate that when restricted to  $\dot F\times \dot F$, $g_\alpha$  determines the resolvent  $(\dot R_\alpha)_{\alpha>0}$ of a normal transition function on $(\dot F,\mathcal{B}(\dot F))$.  In other words,  $\dot R_\alpha$ serves as the resolvent of a normal Markov process $\dot X$ on $\dot F$.

\begin{theorem}\label{THM65}
	There exists a normal transition function $\dot P_t$ on $(\dot F,\mathcal{B}(\dot F))$ whose resolvent is
	\begin{equation}\label{eq:61}
		\dot R_\alpha(x,  A)=\int_A g_{\alpha}(x,y)\fm(dy),\quad x\in \dot F, A\in \mathcal{B}(\dot F).
	\end{equation}
\end{theorem}
\begin{proof}
	Denote by $\widehat{P}_t$ the transition function of the quasidiffusion $\widehat{X}$ on $\widehat{F}$ with speed measure $\widehat{\fm}$.  Note that $\widehat{P}_t(\widehat{x},\cdot)$ is absolutely continuous with respect to $\widehat{\fm}$ for any $\widehat{x}\in \widehat{F}$; see, e.g., \cite{K75}.  Denote its transition density by $\widehat{p}_t(\widehat{x},\cdot)$. Let
	\[
		\dot P_t(x,A):=\widehat{P}_t(\bs(x),\bs(A)),\quad t\geq 0,  x\in \dot F,  A\in \mathcal{B}(\dot F).
	\]
	We have
	\[
		\widehat{P}_t(\bs(x),\bs(A))=\int_{\bs(A)}\widehat{p}_t(\bs(x),\widehat y)\widehat{\fm}(d\widehat{y})=\int_A \widehat{p}_t(\bs(x),\bs(y))\fm(dy).
	\]
	Hence $\dot P_t(x,\cdot)$ is absolutely continuous with respect to $\fm$, and its density function is $\widehat{p}_t(\bs(x),\bs(\cdot))$.
	Clearly $\dot P_t(\cdot,  A)$ is $\mathcal{B}(\dot F)$-measurable,
	\[\dot P_0(x,A)=\widehat{P}_0(\bs(x),\bs(A))=1_{\bs(A)}(\bs(x))=1_A(x),
	\]
	and on account of Lemma~\ref{LM63},  $\dot P_t(x,\cdot)$ is a probability measure on $\dot F$. Additionally,  using Lemma~\ref{LM63} again, we get that for $t,s\geq 0$,
	\[
		\begin{aligned}
			\int_{\dot F}\dot P_t(x,dy)\dot P_s(y,A) & =\int_{\dot F}\widehat{p}_t(\bs(x),\bs(y))\widehat{P}_s(\bs(y),\bs(A))\fm(dy)                                     \\
			                                         & =\int_{\dot{\widehat{F}}}\widehat{p}_t(\bs(x),\widehat y)\widehat{P}_s(\widehat y,\bs(A))\widehat\fm(d\widehat y) \\
			                                         & =\int_{\widehat{F}}\widehat{P}_t(\bs(x), d\widehat{y})\widehat{P}_s(\widehat{y}, \bs(A))                          \\
			                                         & =\widehat{P}_{t+s}(\bs(x),\bs(A))=\dot P_{t+s}(x,A).
		\end{aligned}\]
	Hence the Chapman-Kolmogorov equation is verified.  As a result,  $\dot P_t$ is a normal transition function on $(\dot F,\mathcal B(\dot F))$.

	Let $\widehat{g}_\alpha$ be the reproducing kernel of $\widehat{X}$,  i.e.,  its resolvent density with respect to $\widehat{\fm}$. To obtain \eqref{eq:61}, it suffices to note that $g_\alpha(x,y)=\widehat{g}_\alpha(\bs(x),\bs(y))$ for $x,y\in \dot F$ .  This completes the proof.
\end{proof}

In general, we cannot state that $\dot X$ possesses certain pathwise properties on $\dot F$.  It is even not a \emph{right process} if $\dot{\widehat F}\neq \widehat F$,  as we will show in Corollary~\ref{COR69}.

\subsection{Resolvent extension}
Let $\dot R_\alpha$ be the resolvent obtained in Theorem~\ref{THM65}.
Our goal is to identify extensions of $\dot R_\alpha$ that correspond to certain \emph{nice} Markov processes.  The following terminology is commonly used in this context.

\begin{definition}
	Let $(F_i,\mathcal{B}(F_i))$,  $i=1,2$,  be two topological spaces with Borel $\sigma$-algebras such that $F_1\subset F_2$ and $\mathcal{B}(F_1)=\mathcal{B}(F_2)|_{F_1}:=\{A\cap F_1: A\in \mathcal{B}(F_2))$.  Furthermore, let $(R^1_\alpha)_{\alpha>0}$ be the resolvent of a Markov process on $F_1$.  Then the resolvent $(R^2_\alpha)_{\alpha>0}$ of another Markov process on $F_2$ is called an \emph{extension} of $(R^1_\alpha)_{\alpha>0}$ if
	\begin{itemize}
		\item[(1)] $R^2_\alpha(1_{F_2\setminus F_1})=0$.
		\item[(2)] $(R^2_\alpha f)|_{F_1}=R^1_\alpha (f|_{F_1})$ on $F_1$ for any bounded $f\in \mathcal{B}(F_2)$.
	\end{itemize}
\end{definition}

We start by examining the resolvent $R^*_\alpha$ of $X^*$ on $F^*$, which is the image process of $\widehat{X}$ under the homeomorphism $\bs^*$. Let $E^*$ be equal to $F^*$ when $r$ is not entrance, and $E^*:=F^*\cup \{r\}$ when $r$ is entrance. Then, the transition function of $X^*$ acts on $C_\infty(E^*)$ as a Feller semigroup.  Based on Lemma~\ref{LM63}, we can conclude that $R^*_\alpha$ is an extension of $\dot R_\alpha$ Additionally, its resolvent density $g^*_\alpha$ with respect to $\fm^*$ is the continuous extension of $g_\alpha|_{\dot F\times \dot F}$ to $F^*\times F^*$.

There are also interesting extensions of $\dot R_\alpha$ on $F$. Consider a modified function $\tilde{\bs}:F\rightarrow \widehat{F}$ of $\bs$ as follows. For $x\in \dot F$, we set $\tilde{\bs}(x)=\bs(x)$. For $x\in F\setminus \dot F$ with $\bs(x-) \in \widehat{F}$ (or $\bs(x+)\in \widehat{F}$), we define $\tilde{\bs}(x)$ as $\bs(x-)$ (or $\bs(x+)$). It is important to note that if both $\bs(x-)\in \widehat{F}$ and $\bs(x+)\in \widehat{F}$, we may take either $\tilde{\bs}(x)=\bs(x-)$ or $\tilde{\bs}(x+)=\bs(x+)$. In other words, there may be different candidates for $\tilde{\bs}$.  Note that $\br(\tilde{\bs}(x))=x$ for any $x\in F$. Furthermore, let
\begin{equation}\label{eq:62}	R_\alpha(x,A):=\widehat{R}_\alpha(\tilde{\bs}(x),\tilde{\bs}(A)),\quad x\in F, A\in \mathcal{B}(F).
\end{equation}
Obviously
\[
	R_\alpha(x,A)=\int_A \tilde{g}_\alpha(x,y)\fm(dy),
\]
where $\tilde{g}_\alpha (x,y)=\widehat{g}_\alpha (\tilde{\bs}(x),\tilde{\bs}(y))$.  However, it should be noted that while  $\tilde{g}_\alpha=g_\alpha$ on $\dot F\times \dot F$, $\tilde{g}_\alpha(x,y)\neq g_\alpha(x,y)$ when  either $x\in F\setminus \dot F$ or $y\in F\setminus \dot F$.

\begin{theorem}\label{THM67}
	Let  $(R_\alpha)_{\alpha>0}$ be the family of kernels on $F$ defined as \eqref{eq:62}. Then it is an extension of $(\dot R_\alpha)_{\alpha>0}$,  and there exists a normal c\`adl\`ag $\fm$-symmetric Markov process $X=(X_t)_{t\geq 0}$ on $F$ that satisfies the skip-free property (see Remark~\ref{RM42}) and the quasi-left-continuity, whose resolvent is $(R_\alpha)_{\alpha>0}$.  Furthermore,
	\begin{itemize}
		\item[(1)]  $X^*$ is the canonical Ray-Knight compactification of $X$.
		\item[(2)]  $X$ satisfies the strong Markov property if and only if $\tilde\bs$ is continuous (equivalently, $\tilde{\bs}(F)=\widehat{F}$). In this case, $F=F^*$ and $X$ is identified with $X^*$.
	\end{itemize}
\end{theorem}
\begin{proof}
	Write the quasidiffusion on $\widehat{F}$ with speed measure $\widehat{\fm}$ and lifetime $\widehat \zeta$ as
	\[
		\widehat X=\left\{\widehat{\Omega},\widehat \sF, \widehat{\sF}_t, \widehat{X}_t, (\widehat{\mathbf{P}}_{\widehat{x}})_{\widehat{x}\in \widehat{F}}\right\}.
	\]
	We define $ \Omega:=\left\{\omega\in \widehat{\Omega}: \widehat{X}_0(\omega)\in \tilde\bs(F)\right\}\in \widehat\sF_0, \sF:=\widehat{\sF}\cap \Omega=\{A\cap  \Omega: A\in \widehat{\sF}\}$ and
	\begin{equation}
		\begin{aligned}
			 & \sF_t:=\widehat{\sF}_t\cap \Omega=\{A\cap  \Omega: A\in \widehat{\sF}_t\},\quad  X_t(\omega):=\br(\widehat{X}_t(\omega)), \; \omega\in \Omega, \\
			 & {\mathbf{P}}_x:=\widehat{\mathbf{P}}_{\tilde\bs(x)}|_{ \Omega},\; x\in F \quad  \zeta(\omega):=\widehat{\zeta}(\omega),\; \omega\in \Omega.
		\end{aligned}
	\end{equation}
	Mimicking the arguments in Theorem~\ref{THM51},  one can deduce that $$X:=\{\Omega, \sF, \sF_t,  X_t,  \mathbf{P}_x\}$$ is a normal $\fm$-symmetric Markov process on $F$ with lifetime $\zeta$ and resolvent $R_\alpha$.

	Note that $\br: \widehat{F}\rightarrow F$ is continuous.  Hence $X_t=\br(\widehat{X}_t)$ is c\`adl\`ag because so is $\widehat{X}_t$.  To verify the skip-free property of $X$,  suppose $\widehat{x}:=\widehat{X}_{t-}\in \{\bs(x),\bs(x+),\bs(x-)\}$ and $\widehat{y}:=\widehat{X}_t\in \{\bs(y),\bs(y+),\bs(y-)\}$ for some $x,y\in F$ with $x< y$.  Since $\widehat{X}$ satisfies the skip-free property,  it follows that $(\widehat{x},\widehat{y})\cap \widehat{F}=\emptyset$,  i.e.  $\widehat{\fm}\left((\widehat{x},\widehat{y})\right)=0$.  Using the definition of $\widehat{\fm}$,  one can get $\fm((x,y))=0$.  This implies $(X_{t-},X_t)=(\br(\widehat{x}),\br(\widehat{y}))=(x,y)\subset F^c$.  Particularly,  $X$ satisfies the skip-free property.
	Let us turn to prove the quasi-left-continuity of $X$. Take a sequence of $\sF_t$-stopping times $\sigma_n$ increasing to $\sigma$.  Since $\sF_t\subset \widehat{\sF}_t$,  it follows that $\sigma_n, \sigma$ are also $\widehat{\sF}_t$-stopping times.  On account of the quasi-left-continuity of $\widehat{X}$ and the continuity of $\br:\widehat{F}\rightarrow F$,  we have
	\[
		\lim_{n\rightarrow \infty} X_{\sigma_n}=\lim_{n\rightarrow \infty} \br(\widehat{X}_{\sigma_n})=\br(\widehat{X}_\sigma)=X_{\sigma}.
	\]
	This establishes the quasi-left-continuity of $X$.

	The analogous arguments to the proof of \cite[Theorem~4.10]{L23} yield that $X^*$ is the canonical Ray-Knight compactification of $X$.

	To prove the second assertion for $X$, we first note that the continuity of $\tilde{\bs}$ is equivalent to $\tilde{\bs}(F)=\widehat{F}$,  which also implies that $\tilde{\bs}$ is a homeomorphism.  To demonstrate this,  assume that $\tilde{\bs}:F\rightarrow \widehat{F}$ is continuous.  Denote the continuous extension of $\tilde{\bs}$ to $\bar F:= F\cup \{r\}$  still by $\tilde{\bs}$.  Since $\bar{F}$ is compact in $\overline{\mathbb R}$ and $\tilde{\bs}:\bar{F}\rightarrow \widehat{F}\cup \{\widehat{r}\}$ is continuous and injective, it follows that $\tilde{\bs}(\bar{F})\subset \widehat{F}\cup\{\widehat{r}\}$ is compact. By  Lemma~\ref{LM63}, we have $\widehat{\fm}(\widehat{F}\setminus \tilde{\bs}(\bar F))\leq \widehat{\fm}(\widehat{F}\setminus {\bs}(\dot F))=0$.  Therefore, the definition of $\widehat{F}$ yields $\tilde{\bs}(\bar{F})=\widehat{F}\cup \{\widehat{r}\}$.  In particular,  $\tilde{\bs}:F\rightarrow \widehat{F}$ is a homeomorphism. Conversely,  suppose $\tilde{\bs}(F)=\widehat{F}$.  Argue by contraction and suppose further that $x_n\rightarrow x$ in $F$ while $\tilde{\bs}(x_n)\rightarrow \tilde{\bs}(y)$ in $\widehat{F}$ with $x\neq y$.  Taking a subsequence if necessary,  we may and do assume without loss of generality that $x_n$ is increasing.  Then $\tilde{\bs}(y)=\lim_{n\rightarrow \infty}\tilde{\bs}(x_n)\leq \tilde{\bs}(x)$. Since $\tilde{\bs}$ is strictly increasing and $x\neq y$, we must have $y<x$ . As $x_n\uparrow x>y$ , there exists an integer $N$ such that $x_n\geq x_N>y$ for all $n\geq N$. It follows that $\tilde\bs(x_n)\geq \tilde\bs(x_N)>\bs(y)$.
	This violates the assumption $\tilde{\bs}(x_n)\rightarrow \bs(y)$.

	When $\tilde{\bs}$ is continuous, the previous argument shows that  $\tilde{\bs}$ is a homeomorphism. In particular,  $F=F^*$ and $X$ is identified with $X^*$, satisfying the strong Markov property.

	Conversely, suppose that $X$ satisfies the strong Markov property.  Then $X$ is actually a Hunt process on $F$, and by the first assertion, the Feller process $X^*$ is the canonical Ray-Knight compactification of $X$.  According to \cite[(11.2)]{Ge75},  $F^*\setminus F$ is $\fm^*$-polar with respect to $X^*$.  Since $X^*$ is the homeomorphic image of $\widehat{X}$ and every singleton of $\widehat{F}$ is not $\widehat{\fm}$-polar with respect to $\widehat{X}$,  we deduce that $F^*\setminus F=\emptyset$.  Therefore,  $\tilde\bs(F)=\widehat{F}$.  This completes the proof.
\end{proof}
\begin{remark}
	It is important to emphasize that the continuity of $\tilde{\bs}$ does not necessarily imply $\dot F=F$.   For example, consider $I=[0,2]$,  $\bs$ defined as in \eqref{eq:64}, and $\fm(dx):=1_{[0,1]}(x)dx$.  Then $F=\widehat{F}=[0,1]$,  $\tilde{\bs}$ is a homeomorphism between $F$ and $\widehat{F}$,  while $\dot F=[0,1)\neq F$.
\end{remark}

This theorem readily leads to the following.

\begin{corollary}\label{COR69}
	If $\dot{\widehat{F}}\neq \widehat{F}$, then there is no right process on $\dot F$ whose resolvent is $\dot R_\alpha$.
\end{corollary}
\begin{proof}
	Let $\dot X$ be a right process on $\dot F$.  By embedding $\dot F$ into the compact metric space $F\cup \{r\}$ and repeating the arguments in the proof of \cite[Theorem~4.10]{L23},   the Ray-Knight compactification of $\dot X$ in  the sense of \cite[\S10]{Ge75} is actually $X^*$. In particular,  $F^*\setminus \dot F$ is $\fm^*$-polar with respect to $X^*$ and must therefore be empty. This leads to a contradiction with $\dot{\widehat{F}}\neq \widehat{F}$.  Thus, we have completed the proof.
\end{proof}
\begin{remark}
	Actually, we can choose any $F'\in \mathcal{B}(\dot F)$ such that $\fm(\dot F\setminus F')=0$.  By doing so, we can obtain a resolvent on $F'$ by restricting reproducing kernel $g_\alpha|_{F'\times F'}$. In other words, for $x'\in F'$ and $A'\in \mathcal{B}(F')$, we define $R'_\alpha(x', A'):=\int_{A'}g_\alpha(x',y')\fm(dy')$.  This resolvent corresponds to a Markov process,  which is not a right process on $F'$.   A similar investigation has also been documented in \cite{BCR22}.
\end{remark}

\section{Non-decreasing scale function}\label{SEC7}

We close this paper with some remarks regarding the reproducing kernel for the general pair $(\bs,\fm)$ in Definition~\ref{DEF31}.  These remarks are particularly applicable when $\bs$ is constant over certain intervals.

Let $(a_n,b_n)$, $1\leq n\leq N$ with $N\in \bN\cup\{\infty\}$, be the sequence of at most countably many open sub-intervals of $I$ over which $\bs$ is constant. We assume that the values $C_n$ of $\bs$ on each $(a_n,b_n)$ are distinct, so that every $(a_n,b_n)$ is the largest open interval containing some $x\in (a_n,b_n)$ such that $\bs$ is constant on it.
The reproducing kernel $g_\alpha$ remains the same  as in Definition~\ref{DEF38}. However, for any $f\in \mathcal{B}^+_b(I)$, $$R_\alpha f(\cdot):=\int_I g_\alpha(\cdot, y)f(y)\fm(dy)$$ is constant on each $(a_n,b_n)$. (In fact, $\bs(x)=\bs(y)$ always implies $R_\alpha f(x)=R_\alpha f(y)$.)
To transform $R_\alpha$ (or their certain restrictions) into the resolvent of some Markov process, the approach is to redefine the structure of $I$ and, more precisely, to identify each
\[
	\{x\in I: \bs(x)=C_n\}\quad (\supset (a_n,b_n))
\]
as an abstract single point. Denote these abstract points by $x^\#_n$, $1\leq n\leq N$. We define
\[
	I^\#:=\left(I\setminus \cup_{1\leq n\leq N} \{x\in I: \bs(x)=C_n\}\right)\cup\{x^\#_n: 1\leq n\leq N\}.
\]
As a result, $\bs,  \fm$ and $g_\alpha$ can induce corresponding $\bs^\#, \fm^\#$ on $I^\#$ and $g^\#_\alpha$ on $I^\#\times I^\#$.  We still have $\widehat{\fm}=\fm^\#\circ \bs^{\# -1}$.  Additionally,  $\br$ induces a map $\br^\#: \widehat{I}\rightarrow I^\#$.
Denote still by $\widehat{F}$ the topological support of $\widehat{\fm}$ on $\widehat{I}$. Let
\[
	F^\#:=\br^\#(\widehat{F}),\quad {\dot F}^\#:=\{x^\#\in I^\#: \bs^\#(x^\#)\in \widehat{F}\}\subset F^\#.
\]
The following result then holds.

\begin{lemma}
	It holds that $\fm^\#(I^\#\setminus \dot F^\#)=0$.
\end{lemma}
\begin{proof}
	It suffices to note that for $x^\#\notin \dot F^\#$,  $x^\#\in (\bs^\#)^{-1}(\widehat{I}\setminus \widehat{F})$. Therefore, $\fm^\#(I^\#\setminus \dot F^\#)\leq \fm^\#\left((\bs^\#)^{-1}(\widehat{I}\setminus \widehat{F}) \right)=\widehat{\fm}(\widehat{I}\setminus \widehat{F})=0$.
\end{proof}

Mimicking the arguments in Theorem~\ref{THM65},  one can conclude that
\[
	\dot R^\#_\alpha(x^\#, A^\#):=\int_{A^\#} g^\#_\alpha(x^\#,y^\#)\fm^\#(dy^\#),\quad x^\#\in \dot F^\#, A^\#\in \mathcal{B}(\dot F^\#)
\]
serves as the resolvent of a Markov process on $\dot F^\#$. Note that we may still make a certain completion $I^{\# *}$ of $I^\#$ such that $I^{\# *}$ is homeomorphic to $\widehat{I}$; see,  e.g.,  \cite[\S4.1]{L23}.  Denote by $\widehat{X}$ the quasidiffusion on $\widehat{F}$ with speed measure $\widehat{\fm}$. Under this homeomorphism, the image process $X^{\# *}$ on $I^{\# *}$ of $\widehat{X}$ is a Feller process, and the resolvent of $X^{\# *}$ is an extension of $\dot R^\#_\alpha$.
Furthermore,  $\dot R^\#_\alpha$ admits another resolvent extension $R_\alpha^\#$ on $F^\#$,
whose resolvent density with respect to $\fm^\#$ is an analogous modification of $g^\#_\alpha$ to that in Theorem~\ref{THM67}.

In general, we cannot obtain the pathwise properties of the Markov process with resolvent $\dot R^\#_\alpha$ or $R^\#_\alpha$ because the topology on $I^\#$ is unclear.  However, in certain special cases, the investigation can be similar to that  conducted in Theorems~\ref{THM65} and \ref{THM67}. For instance, if the conditions
\begin{equation}\label{eq:71}
	\bs(a_n)=\bs(a_n+),\quad \bs(b_n-)=\bs(b_n)
\end{equation}
hold for all $1\leq n\leq N$, then
each $x^\#_n$ can be identified with the closed interval $[a_n,b_n]$. Consequently,  $I^\#$ can be treated as an (artificial) interval and $\bs^\#$ as a strictly increasing scale function on $I^\#$. Eventually, everything can be categorized into the cases discussed in the previous two sections. 


\bibliographystyle{gbt7714-numerical}  
\bibliography{ResApp} 

\end{document}